\documentclass[reqno]{amsart}

\usepackage[a4paper, centering]{geometry}

\usepackage{pgf,tikz}\usepackage{mathrsfs}\usetikzlibrary{arrows} 

\usepackage{amssymb,amsmath,latexsym,amsthm,amsfonts}
\usepackage{blindtext}
\usepackage{todonotes}
\usepackage{esint,dsfont}

\def\R{{\mathbb R}}
\def\N{\mathbb{N}}
\def\C{\mathbb{C}}
\def\Z{\mathbb{Z}}
\def\D{\mathbb{D}}
\def\T{\mathbb{T}}
\def\ii{\mathrm{i}}

\newtheorem{prop}{\bf Proposition}[section]
\newtheorem{thm}[prop]{\bf Theorem}
\newtheorem{cor}[prop]{\bf Corollary}
\newtheorem{lem}[prop]{\bf Lemma}
\newtheorem{rmk}[prop]{\it Remark}

\newtheorem*{thmA}{\bf Theorem A}
\newtheorem*{thmB}{\bf Theorem B}
\newtheorem*{thmC}{\bf Theorem C}

\begin{document}

\title{Radial Schur multipliers on some generalisations of trees}

\author{Ignacio Vergara}
%
%
\subjclass[2010]{Primary 46L07; Secondary 05C63, 47B10, 47B35, 30H25}
%
%
%
\keywords{Completely bounded multipliers, infinite graphs, Hankel matrices, trace class operators}

\begin{abstract}
We give a characterisation of radial Schur multipliers on finite products of trees. The equivalent condition is that a certain generalised Hankel matrix involving the discrete derivatives of the radial function is a trace class operator. This extends Haagerup, Steenstrup and Szwarc's result for trees. The same condition can be expressed in terms of Besov spaces on the torus. We also prove a similar result for products of hyperbolic graphs and provide a sufficient condition for a function to define a radial Schur multiplier on a finite dimensional CAT(0) cube complex.
\end{abstract}

\maketitle

\section{Introduction}
For any nonempty set $X$, let $\ell_2(X)$ be the Hilbert space of square-summable complex-valued functions on $X$. For each bounded operator $T\in\mathcal{B}(\ell_2(X))$, we may define its matrix coefficients by
\begin{align*}
T_{x,y}=\langle T\delta_y,\delta_x\rangle,\quad\forall x,y\in X.
\end{align*}
Observe that these coefficients completely determine the operator $T$. We say that a function $\phi:X\times X\to\C$ is a Schur multiplier on $X$ if the map
\begin{align*}
M_\phi : (T_{x,y})_{x,y\in X}\mapsto (\phi(x,y)T_{x,y})_{x,y\in X}
\end{align*}
defines a bounded operator on $\mathcal{B}(\ell_2(X))$. We refer the reader to \cite[Chapter 5]{Pis} for more details. The following result, due essentially to Grothendieck \cite{Gro}, gives a very useful characterisation of Schur multipliers. See \cite[Theorem 5.1]{Pis} for a proof.

\begin{thm}[Grothendieck]\label{ThmGro}
Let $X$ be a nonempty set, $\phi:X\times X\to\C$ a function, and $C\geq 0$ a constant. The following are equivalent:
\begin{itemize}
\item[(i)] The function $\phi$ is a Schur multiplier and $\|M_\phi\|\leq C$.
\item[(ii)] The function $\phi$ is a Schur multiplier and the operator $M_\phi$ is completely bounded, with $\|M_\phi\|_{cb}\leq C$.
\item[(iii)] There exist a Hilbert space $\mathcal{H}$ and bounded functions $P,Q:X\to\mathcal{H}$ such that
\begin{align*}
\phi(x,y)=\langle P(x),Q(y)\rangle,\quad\forall x,y\in X,
\end{align*}
and
\begin{align*}
\left(\sup_{x\in X}\|P(x)\|\right)\left(\sup_{y\in X}\|Q(y)\|\right)\leq C.
\end{align*}
\end{itemize}
\end{thm}

Recall that $T\in\mathcal{B}(\mathcal{H})$ is said to be completely bounded if
\begin{align*}
\|T\|_{cb}=\sup_{n\geq 1}\left\|\text{id}_n\otimes T:M_n(\C)\otimes\mathcal{H}\to M_n(\C)\otimes\mathcal{H}\right\| < \infty.
\end{align*}
Motivated by the equivalence (i)$\iff$(ii) in the previous theorem, we shall write
\begin{align*}
\|\phi\|_{cb}=\|M_\phi\|.
\end{align*}
The set of Schur multipliers on $X$ becomes a Banach space with this norm. Moreover, it is a Banach algebra under pointwise multiplication. Our main interest in the present article is the case when $X$ is (the set of vertices of) an infinite graph. Every connected graph can be endowed with its combinatorial distance $d:X\times X\to\N$. That is, for all $x,y\in X$, $d(x,y)$ is the length of the shortest path between $x$ and $y$. Let $\N$ denote the natural numbers $\{0,1,2,...\}$. We say that a function $\phi:X\times X\to\C$ is radial if there exists $\dot{\phi}:\N\to\C$ such that
\begin{align}\label{phi=dotphi}
\phi(x,y)=\dot{\phi}(d(x,y)),\quad\forall x,y\in X.
\end{align}
Conversely, we say that $\dot{\phi}$ defines a radial function $\phi:X\times X\to\C$ if \eqref{phi=dotphi} holds for $\phi$. We also define the discrete derivatives
\begin{align*}
\mathfrak{d}_1\dot{\phi}(n) &= \dot{\phi}(n)-\dot{\phi}(n+1),\\
\mathfrak{d}_2\dot{\phi}(n) &= \dot{\phi}(n)-\dot{\phi}(n+2),\quad\forall n\in\N,
\end{align*}
and by induction, $\mathfrak{d}_j^{m+1}\dot{\phi}(n)= \mathfrak{d}_j(\mathfrak{d}_j^m\dot{\phi})(n)$ for $j=1,2$ and $m\geq 1$. These higher order derivatives admit the following expression,
\begin{align}\label{hoderiv}
\mathfrak{d}_j^m\dot{\phi}(n)=\sum_{k=0}^m\tbinom{N}{k}(-1)^k\dot{\phi}(n+jk).
\end{align}

Let $\mathcal{T}_d$ be the $d$-homogeneous tree. Recall that a Hankel matrix is an infinite matrix of the form $(a_{i+j})_{i,j\in\N}$, where $(a_n)$ is a sequence of complex numbers. Haagerup, Steenstrup and Szwarc \cite{HaaSteSzw} proved that a function $\dot{\phi}:\N\to\C$ defines a radial Schur multiplier on $\mathcal{T}_d$ ($3\leq d\leq\infty$) if and only if the Hankel matrix
\begin{align*}
H=(\mathfrak{d}_2\dot{\phi}(i+j))_{i,j\in\N}
\end{align*}
belongs to the trace class $S_1(\ell_2(\N))$. See \cite[\S 2.4]{Mur} for details on trace class operators. Moreover, they show that the associated Schur multiplier $\phi$ satisfies
\begin{align*}
\|\phi\|_{cb}=|c_+|+|c_-| + \begin{cases}
\left(1-\frac{1}{d-1}\right)\left\|\left(1-\frac{1}{d-1}\tau\right)^{-1}H\right\|_{S_1}, & \text{if } 3\leq d <\infty,\\
\|H\|_{S_1}, & \text{if } d=\infty,
\end{cases}
\end{align*}
where
\begin{align}\label{c+-}
c_\pm=\frac{1}{2}\lim_{n\to\infty}\dot{\phi}(2n) \pm \frac{1}{2}\lim_{n\to\infty}\dot{\phi}(2n+1),
\end{align}
and $\tau:S_1(\ell_2(\N))\to S_1(\ell_2(\N))$ is defined by $\tau(A)=SAS^*$, where $S$ is the forward shift operator on $\ell_2(\N)$. In particular, for $3\leq d <\infty$,
\begin{align*}
\|\phi\|_{cb} &\geq \frac{d-2}{d} \|H\|_{S_1} + |c_+|+|c_-|.
\end{align*}
Since any tree $\mathcal{T}$ of minimum degree $d\geq 3$ admits isometric embeddings
\begin{align*}
\mathcal{T}_d\hookrightarrow  \mathcal{T} \hookrightarrow  \mathcal{T}_\infty,
\end{align*}
and since, by Theorem \ref{ThmGro}, the restriction of a Schur multiplier to a subset is again a Schur multiplier, a corollary of their result is the following.

\begin{thm}[Haagerup--Steenstrup--Szwarc]\label{thmHSS}
Let $\dot{\phi}:\N\to\C$ be a function. Then $\dot{\phi}$ defines a radial Schur multiplier on any tree of minimum degree $d\geq 3$ if and only if the Hankel matrix
\begin{align*}
H=(\mathfrak{d}_2\dot{\phi}(i+j))_{i,j\in\N}
\end{align*}
is an element of $S_1(\ell_2(\N))$. In that case, the following limits exist
\begin{align*}
\lim_{n\to\infty}\dot{\phi}(2n),\quad \lim_{n\to\infty}\dot{\phi}(2n+1),
\end{align*}
and the corresponding Schur multiplier $\phi$ satisfies
\begin{align}\label{boundsphiH1}
\frac{d-2}{d} \|H\|_{S_1} + |c_+|+|c_-|\leq\|\phi\|_{cb} \leq \|H\|_{S_1} + |c_+|+|c_-|,
\end{align}
where
$c_+$ and $c_-$ are defined as in \eqref{c+-}.
\end{thm}

We extend this result to finite products of trees.

\begin{thmA}
Let $N\geq 1$ and let $\dot{\phi}:\N \to\C$ be a bounded function. Then $\dot{\phi}$ defines a radial Schur multiplier on any product of $N$ infinite trees $T_1,...,T_N$ of minimum degrees $d_1,...,d_N\geq 3$ if and only if the generalised Hankel matrix
\begin{align}\label{genHank}
H=\left(\tbinom{N+i-1}{N-1}^{\frac{1}{2}}\tbinom{N+j-1}{N-1}^{\frac{1}{2}}\mathfrak{d}_2^N\dot{\phi}(i+j)\right)_{i,j\in\N}
\end{align}
is an element of $S_1(\ell_2(\N))$. In that case, the following limits exist
\begin{align*}
\lim_{n\to\infty}\dot{\phi}(2n),\quad \lim_{n\to\infty}\dot{\phi}(2n+1),
\end{align*}
and the corresponding Schur multiplier $\phi$ satisfies
\begin{align}\label{boundsphiH}
\left[\prod_{i=1}^N\frac{d_i-2}{d_i}\right] \|H\|_{S_1} + |c_+|+|c_-|\leq\|\phi\|_{cb} \leq \|H\|_{S_1} + |c_+|+|c_-|,
\end{align}
where $c_+$ and $c_-$ are defined as in \eqref{c+-}.
\end{thmA}

The proof of Theorem A uses the same ideas as \cite{HaaSteSzw}; however, some new considerations must be made in order adapt them to products. Observe that we have added the hypothesis that $\dot{\phi}$ is bounded. In Theorem \ref{thmHSS}, this is a consequence of the fact that the Hankel matrix $H$ is of trace class, but this is no longer true in Theorem $A$, as the function $\dot{\phi}(n)=n$ shows. 

We point out that Theorem A can be stated in terms of Besov spaces. Indeed, by a theorem of Peller \cite{Pel}, the condition $H\in S_1(\ell_2(\N))$ is equivalent to the fact that the analytic function
\begin{align*}
z\mapsto(1-z^2)^N\sum_{n\geq 0}\dot{\phi}(n)z^n
\end{align*}
belongs to the Besov space $B_1^N(\T)$ (see Section \ref{Sect_incl} for a definition).

We also obtain a similar result for products of hyperbolic graphs. Using arguments inspired by \cite{HaaSteSzw} and \cite{Oza}, Mei and de la Salle \cite{MeidlS} showed that a sufficient condition for a function $\dot{\phi}:\N \to\C$ to define a radial Schur multiplier on a hyperbolic graph of bounded degree is that the Hankel matrix
\begin{align*}
H=(\mathfrak{d}_1\dot{\phi}(i+j))_{i,j\in\N}
\end{align*}
belongs to $S_1(\ell_2(\N))$. Moreover, the following estimate holds:
\begin{align*}
\|\phi\|_{cb}\leq C\|H\|_{S_1} +|c|,
\end{align*}
where $c=\lim_{n\to\infty}\dot{\phi}(n)$, and $C$ is a constant depending on the graph, which is given by a construction by Ozawa \cite{Oza}. Furthermore, by the characterisation of radial Herz-Schur multipliers on some free products of groups proved by Wysocza\'nski \cite{Wys}, this condition turns out to be also necessary. This follows from the particular case of the hyperbolic group $(\Z/3\Z)\ast(\Z/3\Z)\ast(\Z/3\Z)$. More precisely, these results together yield the following.

\begin{thm}[Wysocza\'nski, Mei--de la Salle]
Let $\dot{\phi}:\N \to\C$ be a function. Then $\dot{\phi}$ defines a radial Schur multiplier on every hyperbolic graph with bounded degree if and only if the Hankel matrix
\begin{align*}
H=(\mathfrak{d}_1\dot{\phi}(i+j))_{i,j\in\N}
\end{align*}
is an element of $S_1(\ell_2(\N))$. Moreover, in that case, $\dot{\phi}(n)$ converges to some $c\in\C$, and there exists $C>0$ depending only on the graph, such that
\begin{align*}
\|\phi\|_{cb}\leq C\|H\|_{S_1} +|c|.
\end{align*}
\end{thm}

We also extend this characterisation to products.

\begin{thmB}
Let $N\geq 1$ and let $\dot{\phi}:\N \to\C$ be a bounded function. Then $\dot{\phi}$ defines a radial Schur multiplier on every product of $N$ hyperbolic graphs with bounded degree $X_1,...,X_N$  if and only if the generalised Hankel matrix
\begin{align*}
H= \left(\tbinom{N+i-1}{N-1}^{\frac{1}{2}}\tbinom{N+j-1}{N-1}^{\frac{1}{2}}\mathfrak{d}_1^N\dot{\phi}(i+j)\right)_{i,j\in\N}
\end{align*}
is an element of $S_1(\ell_2(\N))$. Moreover, in that case, $\dot{\phi}(n)$ converges to some limit $c\in\C$, and there exists $C>0$ depending only on the graphs $X_1,...,X_N$, such that
\begin{align*}
\|\phi\|_{cb}\leq C\|H\|_{S_1} +|c|.
\end{align*}
\end{thmB}

Observe that, once again, we must make the assumption that $\dot{\phi}$ is bounded. The argument that we use to show that the condition $H\in S_1$ is sufficient is a mix of the proof of Theorem A with the ideas of \cite{MeidlS}. In order to prove that this condition is also necessary, we use some tools from \cite{Wys}, but since we only deal with one particular hyperbolic graph, the proof gets reduced to studying a very specific product of homogeneous trees, and then applying some elements of the proof of Theorem A. Again, the characterisation given by Theorem B is also equivalent to
\begin{align*}
(1-z)^N\sum_{n\geq 0}\dot{\phi}(n)z^n\in B_1^N(\T).
\end{align*}

We also show that, as a consequence of Theorem B, groups acting properly on products of hyperbolic graphs with bounded degrees are weakly amenable. A countable discrete group $\Gamma$ is said to be weakly amenable if there exists a sequence of finitely supported functions $\varphi_n:\Gamma\to\C$ converging pointwise to 1 and such that the functions $\tilde{\varphi}_n:\Gamma\times\Gamma\to\C$ given by $\tilde{\varphi}_n(s,t)=\varphi_n(s^{-1}t)$ are Schur multipliers on $\Gamma$ satisfying
\begin{align*}
\sup_{n}\|\tilde{\varphi}_n\|_{cb} < \infty.
\end{align*}
As far as we know, this result is new; however, we do not know if it allows us to obtain new examples of weakly amenable groups.

Finally, we deal with multipliers on finite dimensional CAT(0) cube complexes. A cube complex $X$ is a polyhedral complex in which each cell is isometric to the Euclidean cube $[0,1]^n$ for some $n\in\N$, and the gluing maps are isometries. The dimension of $X$ is the maximum of all such $n$. The CAT(0) condition is defined in terms of the metric on $X$ induced by the Euclidean metric on each cube. It also admits a combinatorial characterisation proved by Gromov \cite{Grom} by what is sometimes referred to as the \textit{link condition}. However, thanks to a very nice result of Chepoi \cite{Che}, we may define CAT(0) cube complexes as those whose 1-skeleton is a median graph. See Section \ref{Sec_SCCCC} for a definition of median graphs. They generalise trees, in the sense that trees are exactly the 1-dimensional CAT(0) cube complexes. Furthermore, a product of $N$ trees defines an $N$-dimensional CAT(0) cube complex. However, the class of finite dimensional CAT(0) cube complexes is far more general. Indeed, Chepoi and Hagen \cite{CheHag} gave an example of a uniformly locally finite CAT(0) cube complex of dimension 5 that cannot be embedded in a finite product of trees.

CAT(0) cube complexes have been widely studied and remain an object of great interest in geometric group theory. We refer the reader to \cite[\S 2]{GueHig} and the references therein for a presentation close to the spirit of the present paper. 

Observe that Theorem A provides a necessary condition for a function to define a radial Schur multiplier on every $N$-dimensional CAT(0) cube complex; however, we do not know whether this condition is also sufficient. The following result asserts that another (stronger) condition is sufficient.

\begin{thmC}
Let $X$ be (the 0-skeleton of) a CAT(0) cube complex of dimension $N<\infty$. Let $\phi:X\times X\to\C$ be a radial function with $\phi=\dot{\phi}\circ d$, and such that the generalised Hankel matrix
\begin{align*}
H=\left(\tbinom{N+i-1}{N-1}^{\frac{1}{2}}\tbinom{N+j-1}{N-1}^{\frac{1}{2}}\mathfrak{d}_2\dot{\phi}(i+j)\right)_{i,j\in\N}
\end{align*}
belongs to $S_1(\ell_2(\N))$. Then the following limits exist:
\begin{align*}
&\lim_{n\to\infty}\dot{\phi}(2n),& & \lim_{n\to\infty}\dot{\phi}(2n+1),
\end{align*}
and $\phi$ defines a radial Schur multiplier on $X$ of norm at most
\begin{align*}
M\|H\|_{S_1} + |c_+| + |c_-|,
\end{align*}
where $c_+$ and $c_-$ are defined as in \eqref{c+-}, and $M>0$ is a constant depending only on the dimension $N$.
\end{thmC}

Observe that that the matrix $H$ is not the same as that of Theorem A. Indeed, it involves only the first derivative $\mathfrak{d}_2\dot{\phi}$, regardless of the dimension $N$. The reason for this is that the proof of Theorem C uses a construction of Mizuta \cite{Miz} that allows us to adapt the arguments of \cite{HaaSteSzw} in this more general context. In fact, Mizuta used this construction to study a very particular family of radial Schur multipliers, in order to show that groups acting properly on finite dimensional CAT(0) cube complexes are weakly amenable. This was proved independently by Guentner and Higson \cite{GueHig} using uniformly bounded representations.

Once again, Theorem C may be stated in terms of Besov spaces. Namely, if the analytic function
\begin{align*}
(1-z^2)\sum_{n\geq 0}\dot{\phi}(n)z^n
\end{align*}
belongs to $B_1^N(\T)$, then $\dot{\phi}$ defines a radial Schur multiplier on any $N$-dimensional CAT(0) cube complex.

This paper is organised as follows. In Section \ref{sectMulti} we define multi-radial multipliers and establish an intermediate result for finite product of trees. Using this, we prove Theorem A in Section \ref{sectRadTree}. Section \ref{sectSCPHG} is devoted to products of hyperbolic graphs and the \textit{if} part of Theorem B. We also prove the consequence concerning weak amenability. We deal with the \textit{only if} part in Section \ref{sectNCCG} by studying the Cayley graph of $(\Z/3\Z)\ast(\Z/3\Z)\ast(\Z/3\Z)$. Theorem C is proved in Section \ref{Sec_SCCCC}. Finally, in Section \ref{Sect_incl}, we prove the characterisations in terms of Besov spaces, and we show some relations between the conditions in Theorems A, B and C. In particular, we prove that they become stronger as $N$ increases. 

\section{Multi-radial multipliers on products of trees}\label{sectMulti}

In this section we prove a more general form of Theorem A by studying what we will call multi-radial multipliers. These objects will also be useful in the proof of Theorem B. Let $X=X_1\times\cdots\times X_N$ be a product of $N$ graphs. Observe that the combinatorial distance in this case is given by
\begin{align*}
d(x,y)=\sum_{i=1}^Nd_i(x_i,y_i)\quad\forall x,y\in X,
\end{align*}
where $x=(x_i)_{i=1}^N$, $y=(y_i)_{i=1}^N$, and $d_i$ is the combinatorial distance on $X_i$. We say that $\phi:X\times X\to\C$ is a multi-radial function if there exists $\tilde{\phi}:\N^N\to\C$ such that
\begin{align*}
\phi(x,y)=\tilde{\phi}(d(x_1,y_1),...,d(x_N,y_N)),\quad\forall x,y\in X.
\end{align*}
In order to precisely state the main result of this section, we need to fix some notation. For any integer $N\geq 1$, let $[N]$ denote the set $\{1,...,N\}$. For each $I\subset[N]$, we define a vector $\chi^I\in\{0,1\}^N$ by
\begin{align}\label{chi_I}
\chi^I_i=\begin{cases} 1 & \text{if } i\in I \\
0 & \text{if } i\notin I.
\end{cases}
\end{align}
We shall also write, for $n=(n_1,...,n_N)\in\N^N$, $|n|=n_1+\cdots+n_N$. Our goal is to prove the following characterisation of multi-radial multipliers in terms of operators on $\ell_2(\N^N)$.

\begin{prop}\label{Prop_mult_rad}
Let $N\geq 1$ and let $\tilde{\phi}:\N^N \to\C$ be a function such that the limits
\begin{align}\label{l_0l_1}
l_0&=\lim_{\substack{|n|\to\infty \\ |n|\text{ even}}}\tilde{\phi}(n) & l_1&=\lim_{\substack{|n|\to\infty \\ |n|\text{ odd}}}\tilde{\phi}(n)
\end{align}
exist. Then $\tilde{\phi}$ defines a multi-radial Schur multiplier on any product of $N$ infinite trees of minimum degrees $d_1,...,d_N\geq 3$ if and only if the operator $T=(T_{n,m})_{m,n\in\N^N}$ given by
\begin{align}\label{genHank_N}
T_{m,n}=\sum_{I\subset[N]}(-1)^{|I|} \tilde{\phi}(m+n+2\chi^I),\quad\forall m,n\in\N^N
\end{align}
is an element of $\in S_1(\ell_2(\N^N))$. In that case
\begin{align*}
\left[\prod_{i=1}^N\frac{d_i-2}{d_i}\right] \|T\|_{S_1} + |c_+|+|c_-|\leq\|\phi\|_{cb} \leq \|T\|_{S_1} + |c_+|+|c_-|,
\end{align*}
where $c_\pm=l_0\pm l_1$.
\end{prop}

\begin{rmk}
In Proposition \ref{Prop_mult_rad}, we make the assumption that the limits $l_0$ and $l_1$ exist, whereas in Theorem A this is a consequence. The reason for this is that multi-radial functions can have completely different behaviours on each coordinate. Take for example $N=2$ and define $\tilde{\phi}(n_1,n_2)=f(n_1)+g(n_2)$, where $f,g:\N\to\C$ are any bounded functions. Then
\begin{align*}
T_{m,n}&=f(m_1+n_1)+g(m_2+n_2)-f(m_1+n_1+2)-g(m_2+n_2)\\
&\quad-f(m_1+n_1)-g(m_2+n_2+2)+f(m_1+n_1+2)+g(m_2+n_2+2)\\
&=0,
\end{align*}
but for $\tilde{\phi}$ to define a Schur multiplier, $f$ and $g$ should at least satisfy the characterisation given by Theorem \ref{thmHSS}.
\end{rmk}

\subsection{Proof of sufficiency}
We begin with a general observation about bipartite graphs. A graph is said to be bipartite if it does not contain any odd-length cycles. This is the case of products of trees, and more generally, median graphs.

\begin{lem}\label{lemoddeven}
Let $X$ be a connected bipartite graph. Then the function $\chi:X\times X\to\{-1,1\}$ given by
\begin{align*}
\chi(x,y)=(-1)^{d(x,y)}
\end{align*}
is a Schur multiplier of norm 1.
\end{lem}
\begin{proof}
Fix $x_0\in X$ and define $P:X\to\C$ by $P(x)=(-1)^{d(x,x_0)}$. Hence, for all $x,y\in X$,
\begin{align*}
\langle P(x),P(y)\rangle=(-1)^{d(x,x_0)+d(y,x_0)}.
\end{align*}
Observe that $d(x,x_0)+d(y,x_0)+d(x,y)$ is even because it is the length of a cycle. Thus
\begin{align*}
(-1)^{d(x,x_0)+d(y,x_0)}=(-1)^{d(x,y)} = \chi(x,y).
\end{align*}
Therefore, by Theorem \ref{ThmGro}, $\chi$ is a Schur multiplier and $\|\chi\|_{cb}\leq 1$. The other inequality follows from the fact that the supremum norm is bounded above by the cb norm.
\end{proof}

We shall fix $N\geq 1$, and consider, for each $i=1,...,N$, an infinite tree $X_i$ endowed with the combinatorial distance $d_i$. We do not make any assumptions on the degrees. We follow the same strategy as in \cite{HaaSteSzw}. For each $i\in\{1,...,N\}$, fix an infinite geodesic $\omega_0^{(i)}:\N\to X_i$. This means that $\omega_0^{(i)}$ is injective and for each $n\in\N$, $\omega_0^{(i)}(n)$ is adjacent to $\omega_0^{(i)}(n+1)$. Observe that for each $x\in X_i$ there exists a unique infinite geodesic $\omega_x:\N\to X_i$ such that $\omega_x(0)=x$, and which eventually flows with $\omega_0^{(i)}$. The latter may be expressed as $|\omega_x\Delta\omega_0^{(i)}|<\infty$, viewing the geodesics as subsets of $X_i$. Moreover, for any $x,y\in X_i$, let
\begin{align}
k_0&=\min\{k\in\N\, :\, \omega_x(k)\in \omega_y\},\label{k_0m_0}\\
m_0&=\min\{m\in\N\, :\, \omega_y(m)\in \omega_x\},\notag
\end{align}
and observe that $k_0+m_0=d(x,y)$. Then, for all $k,m\in\N$,
\begin{align*}
\omega_x(k)=\omega_y(l) \iff \exists\, j\in\N,\ k=k_0+j,\ m=m_0+j.
\end{align*}

With all these observations, we may now give the proof of the \textit{if} part of Proposition \ref{Prop_mult_rad}.

\begin{lem}\label{Lem_suf_mult_rad}
Let $N\geq 1$ and let $\tilde{\phi}:\N^N \to\C$ be a function such that the limits
\begin{align*}
l_0&=\lim_{\substack{|n|\to\infty \\ |n|\text{ even}}}\tilde{\phi}(n) & l_1&=\lim_{\substack{|n|\to\infty \\ |n|\text{ odd}}}\tilde{\phi}(n)
\end{align*}
exist, and such that the operator $T=(T_{n,m})_{m,n\in\N^N}$ given by
\begin{align*}
T_{n,m}=\sum_{I\subset[N]}(-1)^{|I|} \tilde{\phi}(m+n+2\chi^I),\quad\forall m,n\in\N^N,
\end{align*}
where $\chi^I$ is defined as in \eqref{chi_I}, is an element of $\in S_1(\ell_2(\N^N))$. Then $\tilde{\phi}$ defines a multi-radial Schur multiplier on $X=X_1\times\cdots\times X_N$ of norm at most $\|T\|_{S_1} + |c_+|+|c_-|$.
\end{lem}
\begin{proof}
We treat first the case $l_0=l_1=0$. Let $S_2(\ell_2(\N^N))$ be the space of Hilbert--Schmidt operators on $\ell_2(\N^N)$ (see \cite[\S 2.4]{Mur} for details). Using the polar decomposition, we can find $A,B\in S_2(\ell_2(\N^N))$ such that $T=A^*B$ and $\|T\|_{S_1}=\|A\|_{S_2}\|B\|_{S_2}$. Define now, for each $x=(x_i)_{i=1}^N\in X$,
\begin{align*}
P(x)=\sum_{k_1,...,k_N=0}^\infty \delta_{\omega_{x_1}(k_1)}\otimes\cdots\otimes\delta_{\omega_{x_N}(k_N)}\otimes Be_{(k_1,...,k_N)}
\end{align*}
and
\begin{align*}
Q(x)=\sum_{m_1,...,m_N=0}^\infty \delta_{\omega_{x_1}(m_1)}\otimes\cdots\otimes\delta_{\omega_{x_N}(m_N)}\otimes Ae_{(m_1,...,m_N)},
\end{align*}
where $\{e_n\}_{n\in\N^N}$ is the canonical orthonormal basis of $\ell_2(\N^N)$. Observe that
\begin{align*}
\|P(x)\|^2=\sum_{n\in\N^N}^\infty \| Be_{n}\|^2=\|B\|_{S_2}^2,\quad\forall x\in X.
\end{align*}
Similarly, $\|Q(y)\|^2= \|A\|_{S_2}^2$ for all $y\in X$. Now,
\begin{align*}
\langle P(x),Q(y)\rangle =\sum_{\substack{k_1,...,k_N=0\\ m_1,...,m_N=0}}^{\infty} \left(\prod_{i=1}^N\left\langle\delta_{\omega_{x_i}(k_i)},\delta_{\omega_{y_i}(m_i)}\right\rangle\right)
\left\langle A^*Be_{(k_1,...,k_N)}, e_{(m_1,...,m_N)}\right\rangle.
\end{align*}
Recall that
\begin{align*}
\omega_{x_i}(k_i)=\omega_{y_i}(m_i) \iff \exists\, j_i\in\N,\ k_i=k_0+j_i,\ m_i=m_0+j_i,
\end{align*}
where $k_0$ and $m_0$ are defined as in \eqref{k_0m_0}, and satisfy $k_0+m_0=d_i(x_i,y_i)$. Moreover, in that case, since $T=A^*B$,
\begin{align*}
\langle A^*B e_{(k_1,...,k_N)}, e_{(m_1,...,m_N)}\rangle=\sum_{I\subset[N]}(-1)^{|I|} \tilde{\phi}\left(\vec{d}(x,y)+2(j_1,...,j_N)+2\chi^I\right),
\end{align*}
where
\begin{align*}
\vec{d}(x,y)=(d_1(x_1,y_1),...,d_N(x_N,y_N))\in\N^N.
\end{align*}
Thus
\begin{align*}
\langle P(x),Q(y)\rangle =\sum_{j_1,...,j_N=0}^{\infty} \sum_{I\subset[N]}(-1)^{|I|} \tilde{\phi}\left(\vec{d}(x,y)+2(j_1,...,j_N)+2\chi^I\right).
\end{align*}
Now observe that this expression corresponds to $N$ telescoping series. Indeed, fixing $j_1,...,j_{N-1}$ and defining
\begin{align*}
a_j=\sum_{\substack{I\subset[N] \\ N\notin I}}(-1)^{|I|} \tilde{\phi}\left(\vec{d}(x,y)+2(j_1,...,j_{N-1},j)+2\chi^I\right),
\end{align*}
we obtain
\begin{align*}
\sum_{j_N=0}^{\infty} \sum_{I\subset[N]}(-1)^{|I|} \tilde{\phi}\left(\vec{d}(x,y)+2(j_1,...,j_N)+2\chi^I\right)=\sum_{j=0}^{\infty}a_j-a_{j+1}
\end{align*}
and this equals $a_0$ because we have assumed $l_0=l_1=0$. Hence
\begin{align*}
\langle P(x),Q(y)\rangle=\sum_{j_1,...,j_{N-1}=0}^{\infty} \sum_{\substack{I\subset[N] \\ N\notin I}}(-1)^{|I|} \tilde{\phi}\left(\vec{d}(x,y)+2(j_1,...,j_{N-1},0)+2\chi^I\right).
\end{align*}
Repeating this argument for the variables $j_1,...,j_{N-1}$, we get
\begin{align*}
\langle P(x),Q(y)\rangle =\tilde{\phi}\left(d_1(x_1,y_1),...,d_N(x_N,y_N)\right).
\end{align*}
We conclude that this defines a multi-radial Schur multiplier $\phi$ on $X$, and
\begin{align*}
\|\phi\|_{cb}\leq \|A\|_{S_2}\|B\|_{S_2}=\|T\|_{S_1}.
\end{align*}
Now we drop the assumption $l_0=l_1=0$. Define
\begin{align*}
\tilde{\psi}(n)=\tilde{\phi}(n)-\left(c_++(-1)^{|n|}c_-\right),\quad\forall n\in\N^N,
\end{align*}
where $c_\pm=\frac{1}{2}(l_0 \pm l_1)$. Observe that
\begin{align*}
\lim_{|n|\to\infty}\tilde{\psi}(n)=0,
\end{align*}
and
\begin{align*}
\sum_{I\subset[N]}(-1)^{|I|} \tilde{\psi}(n+2\chi^I) = \sum_{I\subset[N]}(-1)^{|I|} \tilde{\phi}(n+2\chi^I),\quad\forall n\in\N^N.
\end{align*}
Hence, by the previous arguments, we know that $\tilde{\psi}$ defines a multi-radial Schur multiplier on $X$ of norm at most $\|T\|_{S_1}$. Using Lemma \ref{lemoddeven}, we conclude that $\phi$ is a Schur multiplier, and
\begin{align*}
\|\phi\|_{cb} \leq \|T\|_{S_1} + |c_+|+|c_-|.
\end{align*}
\end{proof}

\subsection{Double commuting isometries}

Now we will deal with the \textit{only if} part of Proposition \ref{Prop_mult_rad}. Once again, we follow the strategy of \cite{HaaSteSzw}. Their argument is based on the study of a certain isometry on the $\ell_2$ space of a homogeneous tree. In our case, we need to consider $N$ copies ($N\geq 1$) of that isometry and analyse them together, although they act independently, in some sense. For this purpose, we shall need some preliminaries on double commuting isometries.

Given an isometry $V$ on a Hilbert space $\mathcal{H}$, the Wold--von Neumann theorem gives a decomposition of $\mathcal{H}$ as a direct sum of two subspaces such that, on one of them $V$ acts as a unitary, and on the other it is a unilateral shift. We will need an extension of this result for $N$ double commuting isometries. Given $V_1,...,V_N$ isometries on a Hilbert space $\mathcal{H}$, we say that they double commute if, for all $i,j\in\{1,...,N\}$ with $i\neq j$, $V_i$ commutes with both $V_j$ and $V_j^*$. Let us first state the Wold--von Neumann theorem. For a proof, see e.g. \cite[Theorem 3.5.17]{Mur}.

\begin{thm}[Wold--von Neumann]
Let $V$ be an isometry on a Hilbert space $\mathcal{H}$. Then $\mathcal{H}$ admits a decomposition
\begin{align*}
\mathcal{H}=\mathcal{H}_u\oplus\mathcal{H}_s\ ,
\end{align*}
where 
\begin{align*}
\mathcal{H}_u &= \bigcap_{n\in\N}V^n\mathcal{H}, &
\mathcal{H}_s &=\bigoplus_{n\in\N}V^n\left(\text{Ker } V^*\right).
\end{align*}
\end{thm}

S\l oci\'nski \cite{Slo} generalised this decomposition to a certain type of commuting pairs. In particular, his result applies to double commuting isometric pairs $(V_1,V_2)$. Inductively, we can obtain a similar decomposition for double commuting isometries $V_1,...,V_N$. We present the proof here for the sake of completeness. But first we need to fix some notations. Recall that $[N]$ stands for the set $\{1,...,N\}$. Given $I=\{i_1,...,i_k\}\subset [N]$ with $i_1<\cdots < i_k$, and $p=(p_1,...,p_k)\in\N^I$, put
\begin{align*}
V_I^p=V_{i_1}^{p_1}\cdots V_{i_k}^{p_k}.
\end{align*}
We also write $I^c=[N]\setminus I$.

\begin{prop}\label{thmwald}
Let $V_1,..., V_N$ be double commuting isometries on a Hilbert space $\mathcal{H}$. Define, for each nonempty $I\subset [N]$,
\begin{align*}
\mathcal{W}_I=\bigcap_{i\in I}\text{Ker }V_i^*.
\end{align*}
Then $\mathcal{H}$ admits a decomposition
\begin{align*}
\mathcal{H}=\bigoplus_{I\subset [N]} \mathcal{H}_I,
\end{align*}
where 
\begin{align*}
\mathcal{H}_I=\bigoplus_{p\in\N^I}V_I^p\left(\bigcap_{q\in\N^{I^c}}V_{I^c}^q\mathcal{W}_I\right),
\end{align*}
for all nonempty $I\subsetneq[N]$, and
\begin{align*}
\mathcal{H}_\varnothing &=\bigcap_{q\in\N^N}V_{[N]}^q\mathcal{H}, & \mathcal{H}_{[N]} &=\bigoplus_{p\in\N^N}V_{[N]}^p\left(\mathcal{W}_{[N]}\right).
\end{align*}
\end{prop}
\begin{proof}
We proceed by induction on $N$. For $N=1$, this is exactly the Wold--von Neumann decomposition. Suppose now that we have such a decomposition for some $N\geq 1$ and consider $V_1,..., V_{N+1}$ double commuting isometries on a Hilbert space $\mathcal{H}$. Then, again by the Wold--von Neumann decomposition,
\begin{align*}
\mathcal{H}=\mathcal{H}^u\oplus \mathcal{H}^s,
\end{align*}
where 
\begin{align*}
\mathcal{H}^u=\bigcap_{n\geq 0}V_{N+1}^n\mathcal{H},
\end{align*}
and
\begin{align*}
\mathcal{H}^s=\bigoplus_{n\geq 0}V_{N+1}^n\left(\text{Ker }V_{N+1}^*\right).
\end{align*}
Since $V_1,..., V_{N+1}$ are double commuting isometries, $\mathcal{H}^u$ and $\mathcal{H}^s$ are invariant subspaces for $V_1,...,V_N$ and $V_1^*,...,V_N^*$. Hence, by the induction hypothesis,
\begin{align*}
\mathcal{H}^u &=\bigoplus_{I\subset [N]} \mathcal{H}_I^u, & \mathcal{H}^s &=\bigoplus_{I\subset [N]} \mathcal{H}_I^s.
\end{align*}
For $I\neq\varnothing,[N]$, 
\begin{align*}
\mathcal{H}_I^u &=\bigoplus_{p\in\N^I}V_I^p\left(\bigcap_{q\in\N^{I^c}}V_{I^c}^q\left(\mathcal{H}^u\cap\bigcap_{i\in I}\text{Ker }V_i^*\right)\right)\\
&= \bigoplus_{p\in\N^I}V_I^p\left(\bigcap_{q\in\N^{I^c}}V_{I^c}^q\left(\bigcap_{n\geq 0}V_{N+1}^n\left(\bigcap_{i\in I}\text{Ker }V_i^*\right)\right)\right)\\
&= \bigoplus_{p\in\N^I}V_I^p\left(\bigcap_{q\in\N^{J}}V_{J}^{q}\left(\bigcap_{i\in [N+1]\setminus J}\text{Ker }V_i^*\right)\right),
\end{align*}
where $J=I^c\cup\{N+1\}$. We have also
\begin{align}
\mathcal{H}_I^s &=\bigoplus_{p\in\N^I}V_I^p\left(\bigcap_{q\in\N^{I^c}}V_{I^c}^q\left(\mathcal{H}^s\cap\bigcap_{i\in I}\text{Ker }V_i^*\right)\right)\notag\\
&= \bigoplus_{p\in\N^I}V_I^p\left(\bigcap_{q\in\N^{I^c}}V_{I^c}^q\left(\bigoplus_{n\geq 0}V_{N+1}^n\left(\text{Ker }V_{N+1}^*\right)\cap\bigcap_{i\in I}\text{Ker }V_i^*\right)\right).\label{dec_H_I_s}
\end{align}
Now observe that, if $x\in V_{N+1}^n\left(\text{Ker }V_{N+1}^*\right)\cap\text{Ker }V_i^*$, then $x=V_{N+1}^nu$ with $u\in\text{Ker }V_{N+1}^*$. This yields
\begin{align*}
V_i^*u=V_i^*(V_{N+1}^*)^n x=(V_{N+1}^*)^nV_i^* x=0,
\end{align*}
which implies that $x\in V_{N+1}^n\left(\text{Ker }V_{N+1}^*\cap\text{Ker }V_i^*\right)$. Conversely, if $x\in V_{N+1}^n\left(\text{Ker }V_{N+1}^*\cap\text{Ker }V_i^*\right)$, then the fact that $V_{N+1}$ and $V_i^*$ commute implies that $x\in\text{Ker }V_i^*$. Applying this argument to all $i\in I$, we see from \eqref{dec_H_I_s} that
\begin{align*}
\mathcal{H}_I^s &= \bigoplus_{p\in\N^I}V_I^p\left(\bigcap_{q\in\N^{I^c}}V_{I^c}^q\left(\bigoplus_{n\geq 0}V_{N+1}^n\left(\text{Ker }V_{N+1}^*\cap\bigcap_{i\in I}\text{Ker }V_i^*\right)\right)\right)\\
&= \bigoplus_{p\in\N^K}V_K^p\left(\bigcap_{q\in\N^{I^c}}V_{I^c}^q\left(\bigcap_{i\in K}\text{Ker }V_i^*\right)\right),
\end{align*}
where $K=I\cup\{N+1\}$. This proves the result for $I\neq\varnothing,[N]$. The cases $I=\varnothing$ and $I=[N]$ are simpler and follow analogously.
\end{proof}

\begin{cor}\label{corwald}
Let $V_1,..., V_N$ be double commuting isometries on a Hilbert space $\mathcal{H}$. Then $\mathcal{H}$ admits a decomposition
\begin{align}\label{H=+H_I}
\mathcal{H}=\bigoplus_{I\subset [N]} \mathcal{H}_I,
\end{align}
where 
\begin{align*}
\mathcal{H}_I &\cong \ell_2\left(\N^I\right) \otimes Y_I,
\end{align*}
with $Y_I=\bigcap_{q\in\N^{I^c}}V_{I^c}^q\mathcal{W}_I$ for $0<|I|<N$, $Y_\varnothing=\bigcap_{q\in\N^N}V_{[N]}^q\mathcal{H}$, and $Y_{[N]}=\mathcal{W}_{[N]}$. Moreover, if $i\in I$, then $V_i$ acts as a one coordinate shift on $\ell_2(\N^I)$, and if $i\notin I$, then $V_i$ acts as a unitary on $Y_I$.
\end{cor}
\begin{proof}
By construction, if $i\notin I$, then $V_iV_i^*x=x$ for all $x\in Y_I$. This proves that $V_i$ acts as a unitary on $Y_I$. For the shift part, we shall only consider the case $I=[N]$ since it illustrates well all the other cases. We have 
\begin{align*}
\mathcal{H}_{[N]} =\bigoplus_{p\in\N^N}V_{[N]}^p\left(\mathcal{W}_{[N]}\right).
\end{align*}
Then we can define an isomorphism $\ell_2\left(\N^N\right) \otimes\left(\mathcal{W}_{[N]}\right) \cong \mathcal{H}_{[N]}$ by identifying
\begin{align*}
\delta_{n_1}\otimes\delta_{n_2}\otimes\cdots\otimes\delta_{n_N}\otimes w\ \longleftrightarrow\ V_1^{n_1}\cdots V_N^{n_N} w,
\end{align*}
for all $(n_1,...,n_N)\in\N^N$ and $w\in\mathcal{W}_{[N]}$. Then the action of $V_1$ on $\mathcal{H}_{[N]}$ is given by
\begin{align*}
V_1(\delta_{n_1}\otimes\delta_{n_2}\otimes\cdots\otimes\delta_{n_N}\otimes w) &\longleftrightarrow V_1^{n_1+1}V_2^{n_2}\cdots V_N^{n_N} w\\
&\longleftrightarrow \delta_{n_1+1}\otimes\delta_{n_2}\otimes\cdots\otimes\delta_{n_N}\otimes w,
\end{align*}
which corresponds to the forward shift operator on the first coordinate. For $V_2,...,V_N$, the argument is analogous.
\end{proof}

\subsection{Proof of necessity}
From now on, we shall fix $X=X_1\times\cdots\times X_N$, where $X_i$ is a $(q_i+1)$-regular tree ($2\leq q_i < \infty$, $i=1,...,N$). Our aim is to prove the following result.

\begin{lem}\label{Lem_nec_mult_rad}
Let $\phi:X\times X\to\C$ be a multi-radial Schur multiplier with
\begin{align*}
\phi(x,y)=\tilde{\phi}(d(x_1,y_1),...,d(x_N,y_N)),\qquad\forall\ x,y\in X,
\end{align*}
and such that the limits
\begin{align*}
l_0&=\lim_{\substack{|n|\to\infty \\ |n|\text{ even}}}\tilde{\phi}(n) & l_1&=\lim_{\substack{|n|\to\infty \\ |n|\text{ odd}}}\tilde{\phi}(n)
\end{align*}
exist. Then the operator $T=(T_{n,m})_{m,n\in\N^N}$ given by
\begin{align*}
T_{n,m}=\sum_{I\subset[N]}(-1)^{|I|} \tilde{\phi}(m+n+2\chi^I),\quad\forall m,n\in\N^N,
\end{align*}
where $\chi^I$ is defined as in \eqref{chi_I}, is an element of $\in S_1(\ell_2(\N^N))$ of norm at most
\begin{align*}
\left[\prod_{i=1}^N\frac{q_i+1}{q_i-1}\right]\left(\|\phi\|_{cb}-|c_+|-|c_-|\right),
\end{align*}
where $c_\pm=l_0\pm l_1$.
\end{lem}

Recall that on each $X_i$ we may fix an infinite geodesic $\omega_0^{(i)}:\N\to X_i$, and for each $x\in X_i$ there exists a unique infinite geodesic $\omega_x:\N\to X_i$ such that $\omega_x(0)=x$, and $|\omega_x(\N)\Delta\omega_0^{(i)}(\N)|<\infty$. We define $U_i:\ell_2(X_i)\to\ell_2(X_i)$ by
\begin{align*}
U_i\delta_x=\frac{1}{\sqrt{q_i}}\sum_{\substack{ y\in X_i\\ \omega_y(1)=x}}\delta_y,\quad \forall x\in X_i.
\end{align*}
Since $\langle U_i\delta_x,U_i\delta_y\rangle=0$ if $x\neq y$, this extends to an isometry on $\ell_2(X_i)$ whose adjoint is given by
\begin{align*}
U_i^*\delta_y=\frac{1}{\sqrt{q_i}}\delta_{\omega_y(1)}.
\end{align*}
Consider now the C${}^*$-algebra generated by $U_i$,
\begin{align*}
C^*(U_i)=\overline{\text{span}}\{U_i^{m}(U_i^*)^n\ :\ m,n\in\N\}\subset \ell_2(X_i).
\end{align*}
Let 
\begin{align*}
\mathcal{A}=C^*(U_1)\otimes_{\text{min}}\cdots\otimes_{\text{min}} C^*(U_N)
\end{align*}
denote their minimal tensor product. Since we have an explicit faithful representation of $C^*(U_i)$ on $\ell_2(X_i)$, $\mathcal{A}$ is naturally a subalgebra of $\mathcal{B}(\ell_2(X_1)\otimes\cdots\otimes\ell_2(X_N))\cong\mathcal{B}(\ell_2(X))$ (see \cite[\S 11.3]{KadRin} for details). This implies that $U_1,...,U_N$ extend to double commuting isometries on $\ell_2(X)\cong\ell_2(X_1)\otimes\cdots\otimes\ell_2(X_N)$ in the natural way (with a slight abuse of notation):
\begin{align*}
U_1\delta_{(x_1,...,x_N)}=\left(U_1\delta_{x_1}\right)\otimes\delta_{x_2}\otimes\cdots\otimes\delta_N,
\end{align*}
and so forth. Define now, for $i\in\{1,...,N\}$, $m,n\in\N$,
\begin{align}
U_{i,m,n} &=\left(1-\frac{1}{q_i}\right)^{-1}\left(U_i^{m}(U_i^*)^{n}-\frac{1}{q_i} U_i^*U_i^{m}(U_i^*)^{n}U_i\right)\label{Uimini} \\
&= \begin{cases} \left(1-\frac{1}{q_i}\right)^{-1}\left(U_i^{m}(U_i^*)^{n}-\frac{1}{q_i} U_i^{m-1}(U_i^*)^{n-1}\right) & \text{if } m,n\geq 1 \\
U_i^{m}(U_i^*)^{n} & \text{if }\min\{m,n\}=0,
\end{cases}
\end{align}
and observe that we also have
\begin{align}\label{C(U)gen}
C^*(U_i)=\overline{\text{span}}\{U_{i,m,n}\ :\ m,n\in\N\}.
\end{align}
Since $U_1,...,U_N$ are double commuting operators, there is no ambiguity in defining
\begin{align}\label{Uvec}
U(m,n)=\prod_{i=1}^N U_{i,m_i,n_i}\in\mathcal{B}(\ell_2(X)),
\end{align}
for $m=(m_1,...,m_N)$, $n=(n_1,...,n_N)\in\N^N$.

\begin{lem}\label{lemUneq0}
Let $m,n\in\N^N$ and $x,y\in X$. If
\begin{align*}
\langle U(m,n)\delta_y,\delta_x\rangle \neq 0,
\end{align*}
then
\begin{align*}
d(x_i,y_i)=m_i+n_i,\qquad \forall\,i\in\{1,...,N\}.
\end{align*}
In particular, this implies that $d(x,y)=|m|+|n|$.
\end{lem}
\begin{proof}
Observe that $\langle U(m,n)\delta_y,\delta_x\rangle \neq 0$ if and only if
\begin{align*}
\langle U_{i,m_i,n_i}\delta_{y_i},\delta_{x_i}\rangle \neq 0,\qquad \forall\,i\in\{1,...,N\},
\end{align*}
where $x=(x_1,...,x_N)$ and $y=(y_1,...,y_N)$. So, by \cite[Lemma 2.6]{HaaSteSzw}, we have $d(x_i,y_i)=m_i+n_i$.
\end{proof}

\begin{lem}\label{lemMphiU}
Let $\phi:X\times X\to\C$ be a multi-radial Schur multiplier. Then $\mathcal{A}$ is invariant under $M_{\phi}$. Moreover, for all $m,n\in\N^N$
\begin{align}\label{MphiU}
M_\phi(U(m,n))=\tilde{\phi}(m+n)U(m,n).
\end{align}
\end{lem}
\begin{proof}
Using the fact that $\phi$ is multi-radial, together with Lemma \ref{lemUneq0}, we get
\begin{align*}
\langle M_\phi(U(m,n))\delta_y,\delta_x\rangle 
&= \tilde{\phi}(d(x_1,y_1),...,d(x_N,y_N)) \langle U(m,n)\delta_y,\delta_x\rangle \\
&= \tilde{\phi}(m+n)\langle U(m,n)\delta_y,\delta_x\rangle.
\end{align*}
This proves (\ref{MphiU}). Moreover, by (\ref{C(U)gen}), this also shows that $\mathcal{A}$ is invariant under $M_{\phi}$.
\end{proof}

\begin{lem}\label{lemfphi}
Let $\phi:X\times X\to\C$ be a multi-radial Schur multiplier. There exists a bounded linear functional $f_\phi:\mathcal{A}\to \C$ satisfying
\begin{align}\label{f_phi=phi}
f_\phi(U(m,n))=\tilde{\phi}(m+n),\qquad \forall\,m,n\in\N,
\end{align}
and $\|f_\phi\|\leq\|\phi\|_{cb}$.
\end{lem}
\begin{proof}
By Coburn's theorem (cf. \cite[Theorem 3.5.18]{Mur}), together with \cite[Theorem 3.5.11]{Mur}, there exists a $\ast$-homomorphism $\rho_i:C^*(U_i)\to C(\T)$ such that $\rho_i(U_i)(z)=z$ for all $z\in\T$. Let $\gamma_0:C(\T)\to\C$ be the pure state given by
\begin{align*}
\gamma_0(f)=f(1),\qquad\forall\,f\in C(\T).
\end{align*}
Then $\gamma_i=\gamma_0\circ\rho_i$ is a state on $C^*(U_i)$ and
\begin{align*}
\gamma_i(U_i^m(U_i^*)^n)=1,\qquad\forall\,m,n\in\N.
\end{align*}
Let $\gamma:\mathcal{A}\to \C$ be the product state $\gamma_1\otimes\cdots\otimes\gamma_N$, which is uniquely defined by
\begin{align*}
\gamma(W_1\otimes\cdots\otimes W_2)=\gamma_1(W_1)\cdots\gamma_N(W_N).
\end{align*}
See \cite[\S 11.3]{KadRin} for details. Then
\begin{align*}
\gamma(U_1^{m_1}(U_1^*)^{n_1}\otimes\cdots\otimes U_N^{m_N}(U_N^*)^{n_N})=1,\qquad\forall\,m_1,n_1,...,m_N,n_N\in\N.
\end{align*}
Finally, define $f_\phi:\mathcal{A}\to \C$ by
\begin{align*}
f_\phi(W)=\gamma(M_\phi(W)),\qquad\forall\,W\in \mathcal{A}.
\end{align*}
This map satisfies
\begin{align*}
\|f_\phi\|\leq\|M_\phi\|=\|\phi\|_{cb},
\end{align*}
and by Lemma \ref{lemMphiU},
\begin{align*}
f_\phi(U(m,n))=\tilde{\phi}(m+n)\gamma(U(m,n))=\tilde{\phi}(m+n).
\end{align*}
\end{proof}

Now we need to introduce some notation. Observe that we may write $\ell_2(\N^N)\cong \ell_2(\N_1)\otimes\cdots\otimes\ell_2(\N_N)$, where $\N_i$ is a copy of the natural numbers. If $S_i$ denotes the forward shift operator on $\ell_2(\N_i)$, then we can repeat the previous arguments for $S_i$ instead of $U_i$. In particular, $S_1,...,S_N$ extend to double commuting isometries on $\ell_2(\N^N)$, and we can define $S_{i,m_i,n_i}$ and
\begin{align}\label{Svec}
S(m,n)=\prod_{i=1}^N S_{i,m_i,n_i}\in\mathcal{B}(\ell_2(\N^N)),
\end{align}
as in \eqref{Uimini} and \eqref{Uvec}. We call $S_i$ the forward shift operator on the $i$-th coordinate. Now, for every non-empty $I\subset[N]$ and for every $m\in\N^N$, let $m_I$ be the projection of $m$ on $\N^I$, and define the operator
\begin{align*}
S_I^{m_I}=\prod_{i\in I}S_i^{m_i}\in\mathcal{B}(\ell_2(\N^I)).
\end{align*}
Similarly,
\begin{align*}
(S_I^*)^{m_I}=\prod_{i\in I}(S_i^*)^{m_i}.
\end{align*}
We can define $U_I^{m_I}, (U_I^*)^{m_I}\in \mathcal{B}\left(\ell_2\left(\prod_{i\in I}T_i\right)\right)$ analogously. 

\begin{lem}\label{lem_f=sum_I}
Let $f:\mathcal{A}\to\C$ be a bounded linear functional. Then there exists a family of bounded linear forms $\{f^I\}_{I\subset[N]}$ on $\mathcal{A}$ such that
\begin{align}\label{f=sumfI}
f=\displaystyle\sum_{I\subset[N]}f^I\quad\text{and}\quad \|f\|=\displaystyle\sum_{I\subset[N]}\left\|f^I\right\|.
\end{align}
Moreover, let $m,n\in\N^N$ and $I,J\subset[N]$.
\begin{itemize}
\item[(i)] If $J\cap I =\varnothing$, then
\begin{align*}
f^I\left(U\left(m+k\chi^J,n+k\chi^J\right)\right) = f^I\left(U\left(m,n\right)\right),\quad\forall k\in\N.
\end{align*}
\item[(ii)] If $J\cap I\neq\varnothing$, then
\begin{align*}
f^I\left(U\left(m+k\chi^J,n+k\chi^J\right)\right)\xrightarrow[k\to\infty]{} 0.
\end{align*}
\end{itemize}
\end{lem}
\begin{proof}
Let $(\pi,\mathcal{H})$ be the universal representation of $\mathcal{A}$. Then there exist $\xi,\eta\in\mathcal{H}$ such that
\begin{align*}
f(A)=\langle\pi(A)\xi,\eta\rangle,\qquad \forall\ A \in \mathcal{A},
\end{align*}
and $\|f\|=\|\xi\|\|\eta\|$. On the other hand, since $U_1,...,U_N$ are double commuting isometries, we can use Corollary \ref{corwald} to get a decomposition $\mathcal{H}=\bigoplus_{I\subset[N]}\mathcal{H}_I$ with
\begin{align*}
\mathcal{H}_I &\cong \ell_2\left(\N^I\right) \otimes Y_I.
\end{align*}
And so
\begin{align*}
f(A)=\sum_{I\subset[N]}\langle\pi_I(A)\xi^I,\eta^I\rangle,
\end{align*}
where $\pi_I$ is the restriction of $\pi$ to $\mathcal{H}_I$, and $\xi^I$ (resp. $\eta^I$) is the projection of $\xi$ (resp. $\eta$) on $\mathcal{H}_I$. Define, for each $I\subset[N]$,
\begin{align*}
f^I(A)=\langle\pi_I(A)\xi^I,\eta^I\rangle,\quad\forall A\in\mathcal{A}.
\end{align*}
Hence,
\begin{align*}
\|f\|&\leq \sum_{I\subset[N]}\left\|f^I\right\|\\
& \leq \sum_{I\subset[N]}\left\|\xi^I\right\|\left\|\eta^I\right\|\\
&\leq \left(\sum_{I\subset[N]}\left\|\xi^I\right\|^2\right)^{\frac{1}{2}}\left(\sum_{I\subset[N]}\left\|\eta^I\right\|^2\right)^{\frac{1}{2}}\\
& =\|\xi\|\|\eta\|\\
& = \|f\|.
\end{align*}
and so we obtain \eqref{f=sumfI}. Now recall that $\pi_I(U_i)$ is a unitary for $i\in I^c$. Hence, for all $m,n\in\N^N$ and $I\subset[N]$, $\pi_I(U(m,n))$ equals
\begin{align*}
\left(\prod_{i\in I^c}\pi_I(U_i)^{m_i-n_i}\right) 
\left(\prod_{i\in I}\left(1-\tfrac{1}{q_i}\right)^{-1}\pi_I\left(U_i^{m_i}(U_i^*)^{n_i}-\tfrac{1}{q_i}U_i^*U_i^{m_i}(U_i^*)^{n_i}U_i\right)\right).
\end{align*}
This shows that, if $J\cap I=\varnothing$, then
\begin{align*}
\pi_I(U(m+k\chi^J,n+k\chi^J))= \pi_I(U(m,n)),\quad\forall k\in\N,
\end{align*}
which proves (i). Now suppose that $J\cap I\neq\varnothing$ and observe that, for all $k\geq 1$, $\pi_I(U(m+k\chi^J,n+k\chi^J))$ can be written as the product of the three following expressions
\begin{align*}
&\left(\prod_{i\in I^c}\pi_I(U_i)^{m_i-n_i}\right)\\
& \left(\prod_{i\in I\cap J^c}\left(1-\tfrac{1}{q_i}\right)^{-1}\pi_I\left(U_i^{m_i}(U_i^*)^{n_i}-\tfrac{1}{q_i}U_i^*U_i^{m_i}(U_i^*)^{n_i}U_i\right)\right)\\
& \left(\prod_{i\in I\cap J}\left(1-\tfrac{1}{q_i}\right)^{-1}\pi_I\left(U_i^{m_i+k}(U_i^*)^{n_i+k}-\tfrac{1}{q_i}U_i^{m_i+k-1}(U_i^*)^{n_i+k-1}\right)\right).
\end{align*}
Recall that $\pi_I(U_i)$ is a shift if $i\in I$, which implies that $\pi_I\left(U_i^{m_i+k}(U_i^*)^{n_i+k}\right)$ converges to $0$ in the weak operator topology (even more, this is true for the strong operator topology), hence the same holds for the last term above. Therefore $\pi_I(U(m+k\chi^J,n+k\chi^J))$ converges to $0$ in WOT, which yields
\begin{align*}
f^I(U(m+k\chi^J,n+k\chi^J))=\left\langle\pi_I(U(m+k\chi^J,n+k\chi^J))\xi^I,\eta^I\right\rangle\xrightarrow[k\to\infty]{} 0.
\end{align*}
This proves (ii).
\end{proof}

\begin{lem}\label{lemphi=Tr}
Let $\phi:X\times X\to\C$ be a multi-radial Schur multiplier with
\begin{align*}
\phi(x,y)=\tilde{\phi}(d(x_1,y_1),...,d(x_N,y_N)),\qquad\forall\ x,y\in X,
\end{align*}
and assume that the following limits exist:
\begin{align*}
l_0&=\lim_{\substack{|n|\to\infty \\ |n|\text{ even}}}\tilde{\phi}(n), & l_1&=\lim_{\substack{|n|\to\infty \\ |n|\text{ odd}}}\tilde{\phi}(n).
\end{align*}
Then, there exists a trace-class operator $T\in S_1(\ell_2(\N^N))$ such that for all $m,n\in\N^N$,
\begin{align}\label{phi=c+c+TrSmnT}
\tilde{\phi}(m+n) = c_++(-1)^{|m|+|n|}c_-+\text{Tr}\left(S(m, n) T\right),
\end{align}
where $S(m, n)$ is defined as in \eqref{Svec}, and $c_\pm=\frac{1}{2}( l_0\pm l_1)$. Moreover,
\begin{align*}
\|T\|_{S_1}+|c_+|+|c_-|\leq \|\phi\|_{cb}.
\end{align*}
\end{lem}
\begin{proof}
By Lemma \ref{lemfphi}, There exists a bounded linear function $f_\phi:\mathcal{A}\to\C$ such that
\begin{align}\label{phi(n+m)=f_phi}
\tilde{\phi}(m+n)=f_\phi(U(m,n)),\quad\forall m,n\in\N^N,
\end{align}
and $\|f_\phi\|\leq\|\phi\|_{cb}$. Consider the decomposition
\begin{align*}
f_\phi=\sum_{I\subset[N]}f_\phi^I
\end{align*}
given by Lemma \ref{lem_f=sum_I}. Take $J=[N]$ and observe that \eqref{phi(n+m)=f_phi} together with (i) and (ii) in Lemma \ref{lem_f=sum_I} imply that
\begin{align*}
\lim_{k\to\infty}\tilde{\phi}(m+n+2k\chi^J) = \lim_{k\to\infty}\sum_{I\subset[N]}f_\phi^I(U(m+k\chi^J,n+k\chi^J))
= f_\phi^\varnothing(U(m,n)),
\end{align*}
for all $m,n\in\N^N$. So, by the hypothesis, 
\begin{align*}
f_\phi^\varnothing(U(m,n))=c_++(-1)^{|m|+|n|}c_-.
\end{align*}
Now we shall prove by induction on $|I|\in\{1,...,N-1\}$ that $f_\phi^I=0$. Consider first $I=\{i_0\}$ and $J=I^c$. Observe that $I$ is the only nonempty subset of $[N]$ satisfying $I\cap J=\varnothing$. Then by the same arguments as before,
\begin{align*}
c_++(-1)^{|m|+|n|}c_- = \lim_{k\to\infty}\tilde{\phi}(m+n+2k\chi^J) = f_\phi^\varnothing(U(m,n)) + f_\phi^I(U(m,n)),
\end{align*}
which proves that $f_\phi^I(U(m,n))=0$. Since the vector space spanned by $\{U(m,n)\,|\, m,n\in\N\}$ is dense in $\mathcal{A}$, this shows that $f_\phi^I=0$ for all $I$ with $|I|=1$. Now suppose that this holds for $|I|\in\{1,..,l-1\}$ with $l<N$. Take $\tilde{I}\subset[N]$ with $|\tilde{I}|=l$ and $J=I^c$. Then, again by the same arguments,
\begin{align*}
c_++(-1)^{|m|+|n|}c_- &= \lim_{k\to\infty}\tilde{\phi}(m+n+2k\chi^J)\\
& =\lim_{k\to\infty}\sum_{\substack{I\subset[N] \\ |I|\geq l}}f_\phi^I(U(m+k\chi^J,n+k\chi^J))\\
&= f_\phi^\varnothing(U(m,n)) + f_\phi^{\tilde{I}}(U(m,n)),
\end{align*}
which yields $f_\phi^{\tilde{I}}(U(m,n))=0$. We conclude that
\begin{align*}
\tilde{\phi}(m+n)=c_++(-1)^{|m|+|n|}c_- + f_\phi^{[N]}(U(m,n)),\quad\forall m,n\in\N^N.
\end{align*}
Now recall the notations of Corollary \ref{corwald}, and that $f_\phi^I$ was defined as $\langle\pi_I(\cdot)\xi^I,\eta^I\rangle$ in the proof of Lemma \ref{lem_f=sum_I}. Write
\begin{align*}
\xi^{[N]} &= \sum_{\lambda\in\Lambda}f_\lambda\otimes e_\lambda, & \eta^{[N]} &= \sum_{\lambda\in\Lambda}g_\lambda\otimes e_\lambda,
\end{align*}
where $(e_\lambda)_{\lambda\in\Lambda}$ is an orthonormal basis of $Y_{[N]}$, and $f_\lambda, g_\lambda$ are elements of $\ell_2(\N^N)$ such that
\begin{align*}
\|\xi^{[N]}\|^2 &= \sum_{\lambda\in\Lambda}\|f_\lambda\|^2, & \|\eta^{[N]}\|^2 &= \sum_{\lambda\in\Lambda}\| g_\lambda\|^2.
\end{align*}
We have
\begin{align*}
\langle\pi_{[N]}(U^m(U^*)^n)\xi^{[N]},\eta^{[N]}\rangle 
= \sum_{\lambda\in\Lambda} \left\langle S_{[N]}^m\left(S_{[N]}^*\right)^nf_\lambda,g_\lambda\right\rangle
= \text{Tr}\left(S_{[N]}^m\left(S_{[N]}^*\right)^n T\right),
\end{align*}
where
\begin{align*}
T= \sum_{\lambda\in\Lambda} f_\lambda\odot g_\lambda,
\end{align*}
and $f_\lambda\odot g_\lambda \in S_1(\ell_2(\N^N))$ is the rank 1 operator defined by
\begin{align}\label{rank1op}
(f_\lambda\odot g_\lambda) h = \langle h, g_\lambda\rangle f_\lambda,\qquad\forall\,h\in \ell_2(\N^N).
\end{align}
Hence,
\begin{align*}
f_\phi^{[N]}(U(m,n))=\text{Tr}\left(S(m, n) T\right).
\end{align*}
This proves \eqref{phi=c+c+TrSmnT}. Moreover,
\begin{align*}
\|T\|_{S_1} \leq \sum_{\lambda\in \Lambda} \|f_\lambda\| \|g_\lambda\| 
\leq \left(\sum_{\lambda\in \Lambda} \|f_\lambda\|^2\right)^{\frac{1}{2}} \left(\sum_{\lambda\in \Lambda}\|g_\lambda\|^2\right)^\frac{1}{2} 
= \|\xi^{[N]}\| \|\eta^{[N]}\|.
\end{align*}
On the other hand, since $V=\pi_\varnothing(U_1)$ is a unitary, the C${}^*$-algebra $C^*(V)$ is isomorphic to $C(\sigma(V))$, where $\sigma(V)\subset\T$ is the spectrum of $V$. Hence, by the Riesz representation theorem, there exists a complex measure $\mu$ on $\T$ with supp($\mu)\subset\sigma(V)$ such that
\begin{align*}
\left\langle V^k\xi^{\varnothing},\eta^\varnothing\right\rangle = \int_{\T}z^k\,d\mu(z),\quad\forall k\in\Z,
\end{align*}
and $\|\mu\|\leq \|\xi^{\varnothing}\|\|\eta^\varnothing\|$. Furthermore, let $\nu$ be the complex measure on $\T$ given by $\nu=c_+\delta_1+c_-\delta_{-1}$. Then
\begin{align*}
\int_{\T}z^k\,d\nu(z) &= c_++(-1)^kc_- = f_\phi^{\varnothing}(U_1^k) = \int_{\T}z^k\,d\mu(z),& &\forall k\geq 0,\\
\int_{\T}z^k\,d\nu(z) &= c_++(-1)^kc_- = f_\phi^{\varnothing}((U_1^*)^{-k}) = \int_{\T}z^k\,d\mu(z),& &\forall k < 0,
\end{align*}
which implies that $\mu=\nu$ and therefore $\|\mu\|=|c_+|+|c_-|$. We conclude that
\begin{align*}
\|T\|_{S_1}+|c_+|+|c_-|\leq \|\xi^{[N]}\| \|\eta^{[N]}\| + \|\xi^{\varnothing}\|\|\eta^\varnothing\| = \|f_\phi\| \leq \|\phi\|_{cb}.
\end{align*}
\end{proof}

\begin{lem}\label{lemTT'}
Consider, for each $i=1,...,N$, the operator $\tau_i:\mathcal{B}(\ell_2(\N^N))\to\mathcal{B}(\ell_2(\N^N))$ defined by
\begin{align*}
\tau_i(T)=S_iTS_i^*.
\end{align*}
Then for every $T\in S_1(\ell_2(\N^N))$, the operator
\begin{align*}
T'=\left[\prod_{i=1}^N\left(1-\frac{1}{q_i}\right)^{-1}\left(I-\frac{\tau_i}{q_i}\right)\right] T
\end{align*}
is again an element of $S_1(\ell_2(\N^N))$, and 
\begin{align*}
\|T'\|_{S_1} \leq\left[\prod_{i=1}^N\frac{q_i+1}{q_i-1}\right]\|T\|_{S_1}.
\end{align*}
Moreover, it satisfies
\begin{align}\label{TrT=TrT'}
\text{Tr}\left(S(m, n) T\right) = \text{Tr}\left(S^m(S^*)^n T'\right),\quad\forall m,n\in\N^N.
\end{align}
\end{lem}
\begin{proof}
First observe that $\tau_i$ is an injective $\ast$-homomorphism, hence it is an isometry on $\mathcal{B}(\ell_2(\N^N))$. Furthermore, it is also an isometry on $S_1(\ell_2(\N^N))$. Indeed, if $U|T|$ is the polar decomposition of $T\in S_1(\ell_2(\N^N))$, then $\tau_i(U)\tau_i(|T|)$ is the polar decomposition of $\tau_i(T)$. Therefore
\begin{align*}
\|\tau_i(T)\|_{S_1}=\text{Tr}(\tau_i(|T|))=\text{Tr}(|T|)=\|T\|_{S_1}.
\end{align*}
Thus
\begin{align}
\|T'\|_{S_1} &=\left[\prod_{i=1}^N\left(1-\frac{1}{q_i}\right)^{-1}\right]\left\|\left[\prod_{i=1}^N\left(I-\frac{\tau_i}{q_i}\right)\right]T\right\|_{S_1}\notag\\
&\leq\left[\prod_{i=1}^N\left(1-\frac{1}{q_i}\right)^{-1}\right] \left[\prod_{i=1}^N\left(1+\frac{1}{q_i}\right)\right]\|T\|_{S_1}\notag \\
&=\left[\prod_{i=1}^N\frac{q_i+1}{q_i-1}\right]\|T\|_{S_1}.\label{boundT'}
\end{align}
Finally, in order to obtain \eqref{TrT=TrT'}, we shall prove by induction on $k\in\{1,...,N\}$ that for all $T\in S_1(\ell_2(\N^N))$,
\begin{align}
\text{Tr}&\left(\left[\prod_{i=1}^k\left(S_i^{m_i}(S_i^*)^{n_i}-\frac{1}{q_i} S_i^*S_i^{m_i}(S_i^*)^{n_i}S_i\right)\right] T\right)\notag \\
&\qquad= \text{Tr}\left(\left[\prod_{i=1}^k S_i^{m_i}(S_i^*)^{n_i}\right] \left[\prod_{i=1}^k\left(I-\frac{\tau_i}{q_i}\right)\right] T\right).\label{hypindTr}
\end{align}
Recall the identity $\text{Tr}(AB)=\text{Tr}(BA)$. Then, for $k=1$, we have
\begin{align}
\text{Tr}&\left(\left(S_1^{m_1}(S_1^*)^{n_1}-\frac{1}{q_1} S_1^*S_1^{m_1}(S_1^*)^{n_1}S_1\right) T\right)\notag\\
&\qquad = \text{Tr}\left(S_1^{m_1}(S_1^*)^{n_1}T\right) -\frac{1}{q_1}\text{Tr}\left(S_1^{m_1}(S_1^*)^{n_1}S_1 TS_1^*\right)\notag\\
&\qquad = \text{Tr}\left(S_1^{m_1}(S_1^*)^{n_1}T-\frac{1}{q_1}S_1^{m_1}(S_1^*)^{n_1}\tau_1(T)\right)\notag\\
&\qquad = \text{Tr}\left(S_1^{m_1}(S_1^*)^{n_1}\left(I-\frac{\tau_1}{q_1}\right)T\right).\label{compk=1}
\end{align}
Now suppose that \eqref{hypindTr} is true for some $k\in\{1,...,N-1\}$, and define
\begin{align*}
\tilde{T}=\left(S_{k+1}^{m_{k+1}}(S_{k+1}^*)^{n_{k+1}}-\frac{1}{q_{k+1}} S_{k+1}^*S_{k+1}^{m_{k+1}}(S_{k+1}^*)^{n_{k+1}}S_{k+1}\right)T.
\end{align*}
Then
\begin{align*}
\text{Tr}&\left(\left[\prod_{i=1}^{k+1}\left(S_i^{m_i}(S_i^*)^{n_i}-\frac{1}{q_i} S_i^*S_i^{m_i}(S_i^*)^{n_i}S_i\right)\right] T\right)\\
&\qquad=\text{Tr}\left(\left[\prod_{i=1}^{k}\left(S_i^{m_i}(S_i^*)^{n_i}-\frac{1}{q_i} S_i^*S_i^{m_i}(S_i^*)^{n_i}S_i\right)\right] \tilde{T}\right)\\
&\qquad= \text{Tr}\left(\left[\prod_{i=1}^k S_i^{m_i}(S_i^*)^{n_i}\right] \left[\prod_{i=1}^k\left(I-\frac{\tau_i}{q_i}\right)\right] \tilde{T}\right).
\end{align*}
Repeating the computation \eqref{compk=1} for $S_{k+1}$ instead of $S_1$, this equals
\begin{align*}
\text{Tr}\left(\left[\prod_{i=1}^{k+1} S_i^{m_i}(S_i^*)^{n_i}\right] \left[\prod_{i=1}^{k+1}\left(I-\frac{\tau_i}{q_i}\right)\right] T\right).
\end{align*}
We have proven \eqref{hypindTr}. Setting $k=N$ and multiplying by $\prod_{i=1}^N\left(1-\frac{1}{q_i}\right)^{-1}$, we obtain \eqref{TrT=TrT'}.
\end{proof}

Now we are ready to prove Lemma \ref{Lem_nec_mult_rad}.

\begin{proof}[Proof of Lemma \ref{Lem_nec_mult_rad}]
By Lemmas \ref{lemphi=Tr} and \ref{lemTT'}, there exists $T'\in S_1(\ell_2(\N^N))$ such that
\begin{align*}
\tilde{\phi}(m+n) = c_++(-1)^{|m|+|n|}c_-+\text{Tr}\left(S^m(S^*)^nT'\right),\quad\forall m,n\in\N^N,
\end{align*}
and
\begin{align*}
\|T'\|_{S_1} \leq\left[\prod_{i=1}^N\frac{q_i+1}{q_i-1}\right]\left(\|\phi\|_{cb}-|c_+|-|c_-|\right).
\end{align*}
Hence, for all $m,n\in\N^N$ and $I\subset[N]$,
\begin{align*}
\tilde{\phi}(m+n+2\chi^I)  = c_++(-1)^{|m|+|n|}c_-+\text{Tr}\left(S^{m+\chi^I}(S^*)^{n+\chi^I}T'\right).
\end{align*}
Now observe that for each $k\in\{0,...,N\}$ there are $\binom{N}{k}$ subsets $I\subset[N]$ of cardinality $k$. Thus
\begin{align*}
\sum_{I\subset[N]}(-1)^{|I|}\left(c_++(-1)^{|m|+|n|}c_-\right) = \left(c_++(-1)^{|m|+|n|}c_-\right)\sum_{k=0}^N\tbinom{N}{k}(-1)^{k}=0.
\end{align*}
On the other hand,
\begin{align*}
\text{Tr}\left(S^{m+\chi^I}(S^*)^{n+\chi^I}T'\right) &= \sum_{q\in\N^N} \langle T'\delta_q, S^{n+\chi^I}(S^*)^{m+\chi^I}\delta_q\rangle\\
&= \sum_{q\geq m+\chi^I} \langle T'\delta_q, \delta_{q-m+n}\rangle\\
&= \sum_{q\geq \chi^I} T_{n+q,m+q}',
\end{align*}
where $q\geq \chi^I$ stands for the inequality in each coordinate.
Thus,
\begin{align*}
\sum_{I\subset[N]}(-1)^{|I|} \tilde{\phi}(m+n+2\chi^I)  &= \sum_{I\subset[N]}\sum_{q\geq \chi^I} (-1)^{|I|} T_{n+q,m+q}'\\
&= \sum_{q\in\N^N} T_{n+q,m+q}' \sum_{\substack{I\subset[N] \\ q\geq \chi^I}} (-1)^{|I|}.
\end{align*}
Finally, observe that
\begin{align*}
\sum_{\substack{I\subset[N] \\ q\geq \chi^I}} (-1)^{|I|}=\sum_{k=0}^r \tbinom{r}{k}(-1)^k,
\end{align*}
where $r$ is the cardinality of the set $\{i\in[N]\, :\ q_i>0\}$. Hence, this sum equals $0$ for all $q\in\N^N$, with the exception of $q=0$, for which the sum is $1$. We conclude that
\begin{align*}
\sum_{I\subset[N]}(-1)^{|I|} \tilde{\phi}(m+n+2\chi^I)  = T_{n,m}'.
\end{align*}
\end{proof}

To end this section, we give the proof of Proposition \ref{Prop_mult_rad} using Lemmas \ref{Lem_suf_mult_rad} and \ref{Lem_nec_mult_rad}.

\begin{proof}[Proof of Proposition \ref{Prop_mult_rad}]
Fix $N\geq 1$ and $X_1,...,X_N$ infinite trees of minimum degrees $d_1,...,d_N\geq 3$. Put $X=X_1\times\cdots\times X_N$ and let $\tilde{\phi}:\N^N\to\C$ be a function such that the limits $l_0$ and $l_1$ \eqref{l_0l_1} exist. Suppose first that
\begin{align*}
T=\left(\sum_{I\subset[N]}(-1)^{|I|} \tilde{\phi}(m+n+2\chi^I) \right)_{m,n\in\N^N}
\end{align*}
belongs to $S_1(\ell_2(\N^N))$. Then, by Lemma \ref{Lem_suf_mult_rad}, $\tilde{\phi}$ defines a multi-radial Schur multiplier $\phi$ on $X$, and
\begin{align*}
\|\phi\|_{cb} \leq \|T\|_{S_1} + |c_+|+|c_-|.
\end{align*}
Conversely, assume that $\tilde{\phi}$ defines a Schur multiplier on $X$. Let $\mathcal{T}_d$ be the $d$-regular tree. Then there is an isometric embedding
\begin{align*}
\mathcal{T}_{d_1}\times\cdots\times\mathcal{T}_{d_N} \hookrightarrow X.
\end{align*}
Hence, by restriction, $\tilde{\phi}$ defines a multi-radial Schur multiplier on $\mathcal{T}_{d_1}\times\cdots\times\mathcal{T}_{d_N}$ of norm at most $\|\phi\|_{cb}$. Thus, by Lemma \ref{Lem_nec_mult_rad}, $T$ is an element of $S_1(\ell_2(\N))$ of norm at most
\begin{align*}
\left[\prod_{i=1}^N\frac{d_i}{d_i-2}\right]\left(\|\phi\|_{cb}-|c_+|-|c_-|\right),
\end{align*}
which completes the proof.
\end{proof}

\section{Radial multipliers on finite products of trees}\label{sectRadTree}

In this section, we return to radial multipliers and show how Proposition \ref{Prop_mult_rad} implies Theorem A. First, we prove a general fact about trace-class operators that will provide a relation between the generalised Hankel matrices \eqref{genHank} and \eqref{genHank_N}.

\begin{lem}\label{Lem_NN_N}
Let $N\geq 1$ and let $(a_n)_{n\in\N}$ be a sequence of complex numbers. Define the following matrices
\begin{align*}
H&=\left(\tbinom{N+i-1}{N-1}^{\frac{1}{2}}\tbinom{N+j-1}{N-1}^{\frac{1}{2}}a_{i+j}\right)_{i,j\in\N},&
T&=\left(a_{|m|+|n|}\right)_{m,n\in\N^N}.
\end{align*}
Then $H$ belongs to $S_1(\ell_2(\N))$ if and only if $T$ belongs to $S_1(\ell_2(\N^N))$. Moreover, in that case,
\begin{align*}
\|H\|_{S_1(\ell_2(\N))}=\|T\|_{S_1(\ell_2(\N^N))}.
\end{align*}
\end{lem}
\begin{proof}
Consider the following closed subspace of $\ell_2(\N^N)$,
\begin{align*}
E=\{f\in\ell_2(\N^N)\ :\ \exists\ \dot{f}:\N\to\C,\ f(m)=\dot{f}(|m|)\},
\end{align*}
and define $V:\ell_2(\N)\to E$ by
\begin{align*}
V\delta_i=\tbinom{N+i-1}{N-1}^{-\frac{1}{2}}\sum_{|m|=i}\delta_{m}.
\end{align*}
Since $\tbinom{N+i-1}{N-1}$ is exactly the cardinal of the set $\{m\in\N^N\,:\, |m|=i\}$, and since $\langle V\delta_i,V\delta_j\rangle=0$ for $i\neq j$, $V$ extends to a unitary from $\ell_2(\N)$ to $E$. Suppose first that $T\in S_1(\ell_2(\N^N))$, and observe that both the ranges of $T$ and $T^*$ are contained in $E$. Hence, if $T=U|T|$ is the polar decomposition of $T$, then the same holds for the ranges of $U$ and $|T|$. Thus we can write
\begin{align*}
V^*TV=V^*UVV^*|T|V,
\end{align*}
and this is the polar decomposition of $V^*TV$. Thus
\begin{align*}
\|V^*TV\|_{S_1}=\text{Tr}(V^*|T|V)=\text{Tr}(|T|)=\|T\|_{S_1}.
\end{align*}
Finally,
\begin{align*}
\langle V^*TV\delta_j,\delta_i\rangle &=\langle TV\delta_j,V\delta_i\rangle \\
&= \tbinom{N+i-1}{N-1}^{-\frac{1}{2}}\tbinom{N+j-1}{N-1}^{-\frac{1}{2}}\sum_{|p|=i}\sum_{|q|=j}\langle T\delta_{q},\delta_{p}\rangle\\
&= \tbinom{N+i-1}{N-1}^{-\frac{1}{2}}\tbinom{N+j-1}{N-1}^{-\frac{1}{2}}\sum_{|p|=i}\sum_{|q|=j} a_{|p|+|q|}\\
&= \tbinom{N+i-1}{N-1}^{\frac{1}{2}}\tbinom{N+j-1}{N-1}^{\frac{1}{2}}a_{i+j}.
\end{align*}
Hence $V^*TV=H$, which proves one direction of the equivalence and the equality of the norms. Since $V$ is a unitary, the argument for the other direction is analogous.
\end{proof}

The previous result links the characterisations given by Theorem A and Proposition \ref{Prop_mult_rad}. Since the existence of the limits $l_0$ and $l_1$ is one of the hypotheses of Proposition \ref{Prop_mult_rad}, we need a way to prove that they exist under the assumptions of Theorem A. This will be given by Lemma \ref{lema_nconv} below, which relies on the following elementary fact, whose proof we include for the reader's convenience.

\begin{lem}\label{elem_lem}
Let $(a_n)_{n\in\N}$ be a bounded sequence of complex numbers such that the sequence of differences $(a_n - a_{n+1})$ converges to $a\in\C$. Then $a=0$.
\end{lem}
\begin{proof}
Suppose that $a\neq 0$. Then, there exists $n_0\in\N$ such that
\begin{align*}
|a_{n}-a_{n+1}-a|\leq \frac{|a|}{2},\quad\forall n\geq n_0.
\end{align*}
Thus, for all $k\geq 1$,
\begin{align*}
|a_{n_0}-a_{n_0+k}-ka|\leq \sum_{n=0}^{k-1} |a_{n_0+n}-a_{n_0+n+1}-a| \leq k\frac{|a|}{2},
\end{align*}
which implies that
\begin{align*}
|a_{n_0}-a_{n_0+k}| \geq k\frac{|a|}{2},\quad\forall k\geq 1.
\end{align*}
This contradicts the boundedness of $(a_n)$. Therefore, $a=0$.
\end{proof}

Recall the definition of the discrete derivative for a sequence of complex numbers $(a_n)_{n\in\N}$.
\begin{align*}
(\mathfrak{d}_1^0a)_n &= a_n,\\
(\mathfrak{d}_1^{m+1}a)_n &= (\mathfrak{d}_1^ma)_n-(\mathfrak{d}_1^ma)_{n+1},\quad\forall m,n\in\N.
\end{align*}

\begin{lem}\label{lema_nconv}
Let $m\geq 1$ and let $(a_n)_{n\in\N}$ be a bounded sequence of complex numbers such that
\begin{align*}
\sum_{n\geq 0}\tbinom{m+n-1}{m-1}\left|(\mathfrak{d}_1^m a)_n\right| < \infty.
\end{align*}
Then $(a_n)$ converges.
\end{lem}
\begin{proof}
We proceed by induction on $m$. For $m=1$, we have
\begin{align*}
\sum_{n\geq 0}\left|a_{n}-a_{n+1}\right| < \infty.
\end{align*}
By the triangle inequality, this implies that $(a_n)_{n\in\N}$ is a Cauchy sequence. Hence, it converges. Now suppose that the result is true for some $m\geq 1$, and take $(a_n)_{n\in\N}$ such that
\begin{align}\label{bddsumm+1}
\sum_{n\geq 0}\tbinom{m+n}{m}\left|(\mathfrak{d}_1^{m+1} a)_n\right| < \infty.
\end{align}
In particular, the series
\begin{align*}
\sum_{n\geq 0}\left|(\mathfrak{d}_1^{m} a)_n-(\mathfrak{d}_1^{m} a)_{n+1}\right|
\end{align*}
converges. Therefore, $((\mathfrak{d}_1^{m} a)_n)$ converges. Moreover, since $(\mathfrak{d}_1^{m} a)_n=(\mathfrak{d}_1^{m-1} a)_n-(\mathfrak{d}_1^{m-1} a)_{n+1}$, and since $(a_n)$ is bounded, Lemma \ref{elem_lem} implies that $((\mathfrak{d}_1^{m} a)_n)$ must converge to $0$. Thus
\begin{align*}
(\mathfrak{d}_1^{m} a)_n=\sum_{j\geq 0} (\mathfrak{d}_1^{m} a)_{n+j}-(\mathfrak{d}_1^{m} a)_{n+j+1},\quad\forall n\in\N.
\end{align*}
Hence,
\begin{align*}
\left| (\mathfrak{d}_1^{m} a)_n \right| \leq \sum_{j\geq 0} \left| (\mathfrak{d}_1^{m+1} a)_{n+j}\right|,
\end{align*}
which in turn implies
\begin{align*}
\sum_{n\geq 0}\tbinom{m+n-1}{m-1}\left| (\mathfrak{d}_1^{m} a)_n \right| 
&\leq \sum_{n\geq 0}\tbinom{m+n-1}{m-1}\sum_{j\geq 0} \left| (\mathfrak{d}_1^{m+1} a)_{n+j}\right|\\
&\leq \sum_{n\geq 0}\sum_{j\geq 0}\tbinom{m+n+j-1}{m-1} \left| (\mathfrak{d}_1^{m+1} a)_{n+j}\right|\\
&= \sum_{k\geq 0}(k+1)\tbinom{m+k-1}{m-1} \left| (\mathfrak{d}_1^{m+1} a)_{k}\right|\\
&= \sum_{k\geq 0}\frac{m(k+1)}{m+k}\tbinom{m+k}{m} \left| (\mathfrak{d}_1^{m+1} a)_{k}\right|.
\end{align*}
Since $\frac{m(k+1)}{m+k}\leq m$, by \eqref{bddsumm+1}, this is finite. Therefore, by the induction hypothesis, $(a_n)$ converges.
\end{proof}

\begin{proof}[Proof of Theorem A]
Fix $N\geq 1$ and $X_1,...,X_N$ infinite trees of minimum degrees $d_1,...,d_N\geq 3$. Put $X=X_1\times\cdots\times X_N$ and let $\dot{\phi}:\N\to\C$ be a bounded function. Assume first that the matrix
\begin{align*}
H=\left(\tbinom{N+i-1}{N-1}^{\frac{1}{2}}\tbinom{N+j-1}{N-1}^{\frac{1}{2}}\mathfrak{d}_2^N\dot{\phi}(i+j)\right)_{i,j\in\N}
\end{align*}
belongs to $S_1(\ell_2(\N))$. Then its diagonal defines an element of $\ell_1(\N)$, so, by Lemma \ref{lema_nconv} applied to $a_n=\dot{\phi}(2n)$, the following limit exists
\begin{align*}
\lim_{n\to\infty}\dot{\phi}(2n).
\end{align*}
Moreover, if $S$ denotes the forward shift operator on $\ell_2(\N)$, then $HS$ is also of trace class, and
\begin{align*}
(HS)_{i,j}=\tbinom{N+i-1}{N-1}^{\frac{1}{2}}\tbinom{N+j}{N-1}^{\frac{1}{2}}\mathfrak{d}_2^N\dot{\phi}(i+j+1).
\end{align*}
Since $\binom{N+j-1}{N-1}\leq\binom{N+j}{N-1}$, we can apply again Lemma \ref{lema_nconv} to the sequence $a_n=\dot{\phi}(2n+1)$ to get the existence of
\begin{align*}
\lim_{n\to\infty}\dot{\phi}(2n+1).
\end{align*}
Now define $\tilde{\phi}:\N^N\to\C$ by $\tilde{\phi}(n)=\dot{\phi}(|n|)$. Then
\begin{align*}
\sum_{I\subset[N]}(-1)^{|I|} \tilde{\phi}(m+n+2\chi^I)
&=\sum_{I\subset[N]}(-1)^{|I|} \dot{\phi}\left(|m|+|n|+2|I|\right)\\
&=\sum_{k=0}^N\tbinom{N}{k}(-1)^{k} \dot{\phi}\left(|m|+|n|+2k\right)\\
&= \mathfrak{d}_2^N \dot{\phi}\left(|m|+|n|\right).
\end{align*}
So by Lemma \ref{Lem_NN_N}, the operator
\begin{align*}
T=\left(\sum_{I\subset[N]}(-1)^{|I|} \tilde{\phi}(m+n+2\chi^I)\right)_{m,n\in\N^N}
\end{align*} 
belongs to $S_1(\ell_2(\N^N))$. Hence, the (multi-)radial function $\phi(x,y)=\dot{\phi}(d(x,y))$ satisfies the hypotheses of Proposition \ref{Prop_mult_rad}, and therefore it is a Schur multiplier satisfying
\begin{align}\label{estimTphiT}
\left[\prod_{i=1}^N\frac{d_i-2}{d_i}\right] \|T\|_{S_1} + |c_+|+|c_-|\leq\|\phi\|_{cb} \leq \|T\|_{S_1} + |c_+|+|c_-|,
\end{align}
where
\begin{align*}
c_\pm=\frac{1}{2}\lim_{n\to\infty}\dot{\phi}(2n) \pm \frac{1}{2}\lim_{n\to\infty}\dot{\phi}(2n+1).
\end{align*}
Conversely, assume that $\dot{\phi}$ defines a radial Schur multiplier on $X$. By restriction, it defines a radial Schur multiplier on the tree $X_1$, so by Theorem \ref{thmHSS}, the limits $\lim_{n\to\infty}\dot{\phi}(2n)$ and $\lim_{n\to\infty}\dot{\phi}(2n+1)$ exist. Therefore, by Proposition \ref{Prop_mult_rad} and the same computations as before, the operator $T$ is an element of $S_1(\ell_2(\N^N))$ satisfying the estimates \eqref{estimTphiT}.
\end{proof}

\section{Sufficient condition for a product of hyperbolic graphs}\label{sectSCPHG}

Now we turn to products of hyperbolic graphs and prove the \textit{if} part of Theorem B. We also show that groups acting properly on such graphs are weakly amenable.

\subsection{Sufficient condition}
We shall fix now $N\geq 1$ and a family $X_1,...,X_N$ of hyperbolic graphs with bounded degree. Define $X=X_1\times\cdots\times X_N$. The following result proves the first part of Theorem B.

\begin{prop}\label{prop_suff_prod_hyp}
Let $\dot{\phi}:\N\to\C$ be a bounded function such that the generalised Hankel matrix
\begin{align*}
H=\left(\tbinom{N+i-1}{N-1}^{\frac{1}{2}}\tbinom{N+j-1}{N-1}^{\frac{1}{2}}\mathfrak{d}_1^N\dot{\phi}(i+j)\right)_{i,j\in\N}
\end{align*}
belongs to $S_1(\ell_2(\N))$. Then $\dot{\phi}$ defines a radial Schur multiplier $\phi$ on $X$. Moreover, $\dot{\phi}(n)$ converges to some $c\in\C$, and there exists $C>0$ depending only on $X$, such that
\begin{align*}
\|\phi\|_{cb}\leq C\|H\|_{S_1} +|c|.
\end{align*}
\end{prop}
As in \cite{MeidlS}, the main tool that we will use is this remarkable construction by Ozawa.

\begin{thm}{\cite[Proposition 10]{Oza}}\label{thmOzawa}
Let $X$ be a hyperbolic graph with bounded degree. Then there is a Hilbert space $\mathcal{H}$, a constant $C_0>0$ and functions $\eta_k^+,\eta_k^-:X\to\mathcal{H}$ such that
\begin{itemize}
\item[a)] $\left\langle\eta_k^\pm(x),\eta_l^\pm(x)\right\rangle=0$ for all $x\in X$ and $k,l\in\N$ such that $|k-l|\geq 2$.
\item[b)] $\left\|\eta_k^\pm(x)\right\|^2\leq C_0$ for all $x\in X$ and all $k\in\N$.
\item[c)] For all $n\in\N$ and all $x,y\in X$,
\begin{align*}
\sum_{k=0}^n\left\langle\eta_k^+(x),\eta_{n-k}^-(y)\right\rangle
=\left\{\begin{array}{ll}
1 & \text{if } d(x,y)\leq n\\ 0 & \text{otherwise}.
\end{array}\right.
\end{align*}
\end{itemize}
\end{thm}

Like we did in Section \ref{sectMulti}, we shall obtain first a more general result involving multi-radial multipliers.

\begin{lem}\label{Lem_multi_hyp}
Let $N\geq 1$ and let $\tilde{\phi}:\N^N \to\C$ be a function such that the limit
\begin{align*}
c=\lim_{|n|\to\infty}\tilde{\phi}(n)
\end{align*}
exists, and such that the operator $T=(T_{n,m})_{m,n\in\N^N}$ given by
\begin{align*}
T_{n,m}=\sum_{I\subset[N]}(-1)^{|I|} \tilde{\phi}(m+n+\chi^I),\quad\forall m,n\in\N^N,
\end{align*}
where $\chi^I$ is defined as in \eqref{chi_I}, is an element of $\in S_1(\ell_2(\N^N))$. Then $\tilde{\phi}$ defines a multi-radial Schur multiplier $\phi$ on $X$, and there exists $C>0$ depending only on $X$, such that
\begin{align*}
\|\phi\|_{cb}\leq C\|T\|_{S_1} +|c|.
\end{align*}
\end{lem}
\begin{proof}
The proof follows the same lines as that of Lemma \ref{Lem_suf_mult_rad}. Assume first that $c=0$. Take $A,B\in S_2(\ell_2(\N^N))$ such that $T=A^*B$ and $\|T\|_{S_1}=\|A\|_{S_2}\|B\|_{S_2}$. Consider, for each $i=1,...,N$, the functions $\eta_k^\pm$ given by Theorem \ref{thmOzawa}. Observe that these functions are not the same for different hyperbolic graphs; however, we shall make no distinction in the notation since they are defined in different spaces and will not interact with each other. Furthermore, we can let $C_0$ be the maximum of the $N$ constants given by the theorem. Define now, for each $x=(x_i)_{i=1}^N\in X$,
\begin{align*}
P(x)=\sum_{k_1,...,k_N=0}^\infty \eta_{k_1}^+(x_1)\otimes\cdots\otimes\eta_{k_N}^+(x_N)\otimes Be_{(k_1,...,k_N)}
\end{align*}
and
\begin{align*}
Q(x)=\sum_{l_1,...,l_N=0}^\infty \eta_{l_1}^-(x_1)\otimes\cdots\otimes\eta_{l_N}^-(x_N)\otimes Ae_{(l_1,...,l_N)},
\end{align*}
where $\{e_n\}_{n\in\N^N}$ is the canonical orthonormal basis of $\ell_2(\N^N)$. Observe that, for all $x\in X$,
\begin{align*}
P(x)=\sum_{j_1,...,j_N=0}^1\sum_{m_1,...,m_N=0}^\infty \eta_{2m_1+j_1}^+(x_1)\otimes \cdots\otimes &\eta_{2m_N+j_N}^+(x_N)\\
&\otimes Be_{(2m_1+j_1,...,2m_N+j_N)}.
\end{align*}
Then, using  parts (a) and (b) of Theorem \ref{thmOzawa},
\begin{align*}
\|P(x)\|^2 
&\leq 2^N\sum_{j_1,...,j_N=0}^1\left\|\sum_{m_1,...,m_N=0}^\infty \eta_{2m_1+j_1}^+(x_1)\otimes\cdots\otimes\eta_{2m_N+j_N}^+(x_N)\right.\\
&\qquad\qquad\qquad\qquad\qquad\qquad\qquad\qquad\qquad\otimes  Be_{(2m_1+j_1,...,2m_N+j_N)}\Bigg\|^2\\
&=2^N\sum_{j_1,...,j_N=0}^1\sum_{m_1,...,m_N=0}^\infty \left\|\eta_{2m_1+j_1}^+(x_1)\right\|^2\cdots\left\|\eta_{2m_N+j_N}^+(x_N)\right\|^2\\
&\qquad\qquad\qquad\qquad\qquad\qquad\qquad\qquad\qquad \left\|Be_{(2m_1+j_1,...,2m_N+j_N)}\right\|^2\\
&\leq (2C_0)^N\sum_{n\in\N^N}^\infty \left\|Be_n\right\|^2\\
&= (2C_0)^N\|B\|_{S_2}^2.
\end{align*}
Similarly, we obtain $\|Q(y)\|^2\leq (2C_0)^N \|A\|_{S_2}^2$ for all $y\in X$. Now, we may write $\langle P(x),Q(y)\rangle$ as
\begin{align*}
\sum_{n_1=0}^\infty\sum_{m_1=0}^{n_1}\cdots\sum_{n_N=0}^\infty\sum_{m_N=0}^{n_N}
\langle\eta_{m_1}^+(x_1),\eta_{n_1-m_1}^-(&y_1)\rangle\cdots\langle\eta_{m_N}^+(x_N),\eta_{n_N-m_N}^-(y_N)\rangle\\
&\langle Be_{(m_1,...,m_N)},Ae_{(n_1-m_1,...,n_N-m_N)}\rangle.
\end{align*}
So, using Theorem \ref{thmOzawa}(c) and the fact that $T=A^*B$,
\begin{align*}
\langle P(x),Q(y)\rangle =\sum_{\substack{n_1=d(x_1,y_1)\\ \vdots\\ n_N=d(x_N,y_N)}}^\infty  \sum_{I\subset[N]}(-1)^{|I|}\tilde{\phi}((n_1,...,n_N)+\chi^I)
\end{align*}
By the same inductive argument as in the proof of Lemma \ref{Lem_suf_mult_rad}, together with the fact that $c=0$, one shows that this equals $\tilde{\phi}(d(x_1,y_1),...,d(x_N,y_N))$. We conclude that $\tilde{\phi}$ defines a multi-radial Schur multiplier $\phi$ such that 
\begin{align*}
\|\phi\|_{cb}\leq (2C_0)^N \|A\|_{S_2}\|B\|_{S_2}\leq (2C_0)^N\|T\|_{S_1}.
\end{align*}
In the general case, we use the previous argument for $\phi-c$ and conclude that
\begin{align*}
\|\phi\|_{cb}\leq (2C_0)^N\|T\|_{S_1} +|c|.
\end{align*}
\end{proof}

\begin{proof}[Proof of Proposition \ref{prop_suff_prod_hyp}]
Since $H$ is of trace class, its diagonal belongs to $\ell_1(\N)$, which implies, by Lemma \ref{lema_nconv}, that $c=\lim_n \dot{\phi}(n)$ exists. By Lemma \ref{Lem_NN_N}, the multi-radial function $\tilde{\phi}(n)=\dot{\phi}(|n|)$ satisfies the hypotheses of Lemma \ref{Lem_multi_hyp}, which yields the conclusion.
\end{proof}

\subsection{Bounded sequences of multipliers}
To end this section, we show how Proposition \ref{prop_suff_prod_hyp} allows us to prove that groups acting properly on finite products of hyperbolic graphs of bounded degree are weakly amenable. The idea of the proof was essentially devised by Haagerup \cite{Haa} for the free group $\mathbb{F}_2$, and it was later exploited in \cite{Oza} and \cite{Miz} for hyperbolic groups and CAT(0) cubical groups respectively. 

First, we give a proof of the formula \eqref{hoderiv}.

\begin{lem}\label{lemsumbinom}
Let $(a_n)_{n\in\N}$ be a sequence of complex numbers, then for all $n\in\N$ and $m\geq 1$,
\begin{align*}
(\mathfrak{d}_1^m a)_n &= \sum_{j=0}^{m}\tbinom{m}{j}(-1)^j a_{n+j}.
\end{align*}
\end{lem}
\begin{proof}
We proceed by induction on $m$. For $m=1$ it is just the definition of $\mathfrak{d}_1a$. Now suppose that the formula holds for some $m\geq 1$. Then
\begin{align*}
(\mathfrak{d}_1^{m+1} a)_n &= (\mathfrak{d}_1^{m}(\mathfrak{d}_1a))_n\\
&=\sum_{j=0}^{m}\tbinom{m}{j}(-1)^j\left(a_{n+j}-a_{n+j+1}\right)\\
&= \sum_{j=0}^{m}\tbinom{m}{j}(-1)^ja_{n+j} - \sum_{j=1}^{m+1}\tbinom{m}{j-1}(-1)^{j-1}a_{n+j}\\
&= a_n+ \sum_{j=1}^{m}(-1)^j\left(\tbinom{m}{j}+\tbinom{m}{j-1}\right)a_{n+j} - (-1)^{m}a_{n+m+1}\\
&= a_n + \sum_{j=1}^{m}(-1)^j\tbinom{m+1}{j}a_{n+j} + (-1)^{m+1}a_{n+m+1}\\
&= \sum_{j=0}^{m+1}(-1)^j\tbinom{m+1}{j}a_{n+j}.
\end{align*}
\end{proof}

As before, we fix $N\geq 1$ and $X$ a product of $N$ hyperbolic graphs with bounded degrees.

\begin{lem}
For all $r\in(0,1)$, the function $\phi_r(x,y)=r^{d(x,y)}$ is a Schur multiplier on $X$. Moreover, $\|\phi_r\|_{cb}\leq C$, where $C$ is the constant given by Proposition \ref{prop_suff_prod_hyp}.
\end{lem}
\begin{proof}
First observe that
\begin{align*}
\sum_{k=0}^N\tbinom{N}{k}(-1)^kr^{i+j+k} = (1-r)^Nr^{i+j},
\end{align*}
for all $i,j\in\N$. Define $f\in\ell_2(\N)$ by $f(i)=\tbinom{N-1+i}{N-1}^{\frac{1}{2}}r^{i}$, and let $H$ be the positive rank-1 operator given by $H=(1-r)^Nf\odot f$, which is defined as in \eqref{rank1op}. Then we have
\begin{align*}
H_{i,j}=\tbinom{N-1+i}{N-1}^{\frac{1}{2}}\tbinom{N-1+j}{N-1}^{\frac{1}{2}}(1-r)^Nr^{i+j},\quad\forall i,j\in\N,
\end{align*}
and
\begin{align*}
\|H\|_{S_1}=\text{Tr}(H)=(1-r)^N\sum_{j\geq 0}\tbinom{N-1+j}{N-1}r^{2j}.
\end{align*}
Observe that $\displaystyle\sum_{j\geq 0}\tbinom{N-1+j}{N-1}z^{2j}$ is the power series around 0 of the analytic function $(1-z^2)^{-N}$. Hence
\begin{align*}
\|H\|_{S_1}=\frac{(1-r)^N}{(1-r^2)^N}=\frac{1}{(1+r)^N}< 1.
\end{align*}
The result follows from Proposition \ref{prop_suff_prod_hyp}.
\end{proof}

\begin{lem}
For every $n\in\N$, define $\varphi_n:X\times X\to\C$ by
\begin{align*}
\varphi_n(x,y)=\begin{cases}
1& \text{if } d(x,y)=n\\
0& \text{otherwise}.
\end{cases}
\end{align*}
There exists a constant $C>0$ such that, for every $n\in\N$,
\begin{align}\label{linearbound}
\|\varphi_n\|_{cb}\leq C(n+1)^N.
\end{align}
\end{lem}
\begin{proof}
Write $\varphi_n(x,y)=\dot{\varphi}_n(d(x,y))$. Then
\begin{align}\label{sum_varphi_n}
\sum_{k=0}^N\tbinom{N}{k}(-1)^k\dot{\varphi}_n(i+j+k)=\begin{cases}
\tbinom{N}{n-i-j}(-1)^{n-i-j} & \text{if } n-N\leq i+j\leq n\\
0 & \text{otherwise}.
\end{cases}
\end{align}
Define
\begin{align*}
H_n= \left((1+i+j)^{N-1}\mathfrak{d}_1^N\dot{\varphi}_n(i+j)\right)_{i,j\in\N},
\end{align*}
and let $D^{(l)}$ ($l\in\N$) be the Hankel matrix given by
\begin{align*}
D^{(l)}_{i,j}=\begin{cases}
1 & \text{if } i+j=l\\
0 & \text{otherwise}.
\end{cases}
\end{align*}
Then, by \eqref{sum_varphi_n}, $H_n$ belongs to the linear span of $\left(D^{(l)}\right)_{l\in\N}$. In particular, for $n\geq N$,
\begin{align*}
H_n=\displaystyle\sum_{l=n-N}^n(l+1)^{N-1} \tbinom{N}{n-l}(-1)^{n-l}D^{(l)}.
\end{align*}
Observe that $\left(D^{(l)}\right)^*D^{(l)}$ is a diagonal matrix whose first $l+1$ diagonal entries are 1 and the rest are 0. Hence
\begin{align*}
\|D^{(l)}\|_{S_1}=\text{Tr}\left(\left(\left(D^{(l)}\right)^*D^{(l)}\right)^{\frac{1}{2}}\right)=\text{Tr}\left(\left(D^{(l)}\right)^*D^{(l)}\right)=l+1.
\end{align*}
Therefore,
\begin{align*}
\|H\|_{S_1}&\leq \sum_{k=0}^N \tbinom{N}{k}(n-k+1)^N\\
&\leq (n+1)^N\sum_{k=0}^N \tbinom{N}{k}\\
&= 2^N(1+n)^N.
\end{align*}
This, together with Lemma \ref{lemH123} (with $\alpha+\beta=N-1$) and Proposition \ref{prop_suff_prod_hyp}, proves the result for $n\geq N$. Taking $C$ big enough, we obtain the estimate \eqref{linearbound} for all $n$.
\end{proof}

These two results have the following consequence. For a proof, see e.g. \cite[Theorem 3]{Miz}.

\begin{cor}
Let $\Gamma$ be a countable discrete group acting properly by isometries on a finite product of hyperbolic graphs with bounded degrees. Then $\Gamma$ is weakly amenable.
\end{cor}

\begin{rmk}
Using Theorem A instead of Proposition \ref{prop_suff_prod_hyp}, the same arguments show that if a group acts properly on a finite product of trees, then it is weakly amenable. Here we do not make any assumptions on the degrees.
\end{rmk}

\section{Necessary condition for products of the Cayley graph of $(\Z/3\Z)\ast(\Z/3\Z)\ast(\Z/3\Z)$}\label{sectNCCG}

Now we shall finish the proof of Theorem B by studying a very particular hyperbolic graph. The main tool here is the tree of Serre. For any group which is a free product of other groups, Serre \cite{Ser} constructed a certain tree with some very nice properties. We describe it now, following the presentation of \cite[\S 4]{Wys}.

Let $(G_i)_{i\in I}$ be a family of groups and let $G=\ast_{i\in I} G_i$ be their free product. We define the tree $\Gamma(G)$ in the following way:
\begin{itemize}
\item[(i)] The set of vertices $X$ consists of two disjoints subsets $X_0$ and $X_1$, where
\begin{align*}
X_0 & =G,
\\ X_1 & = \{gG_i\, :\, g\in G,\ i\in I\}.
\end{align*}
So the vertices are all the elements of $G$ and all the cosets with respect to the subgroups $G_i$, $i\in I$.
\item[(ii)] The edges are all the pairs $\{g, gG_i\}$, with $g\in G$ and $i\in I$.
\end{itemize}
This implies that the elements of $X_0$ have all degree $|I|$, and an element $gG_i\in X_1$ has degree $|G_i|$. From now on, we shall fix $I=\{1,2,3\}$, and $G_i=\Z/3\Z$ for all $i$. Hence,
\begin{align*}
G=(\Z/3\Z)\ast(\Z/3\Z)\ast(\Z/3\Z).
\end{align*}
Observe that, in this case, $\Gamma(G)$ is the 3-homogeneous tree $\mathcal{T}_3$. Let $X$ be the Cayley graph of $G$ with generating set $G_1\cup G_2 \cup G_3$. Then $X$ is a hyperbolic graph. Let $d$ and $d_{\Gamma(G)}$ denote the combinatorial distances of $X$ and $\Gamma(G)$ respectively. Then, the previous construction gives a map $\Psi:X\hookrightarrow\Gamma(G)$ satisfying
\begin{align}\label{dGamma=2d}
d_{\Gamma(G)}(\Psi(x),\Psi(y))=2 d(x,y),\quad\forall x,y\in X.
\end{align}
Furthermore, let $f:\Gamma(G)\to\Gamma(G)$ be the automorphism described as follows. We view the empty word $e$ as the root of $\Gamma(G)$:
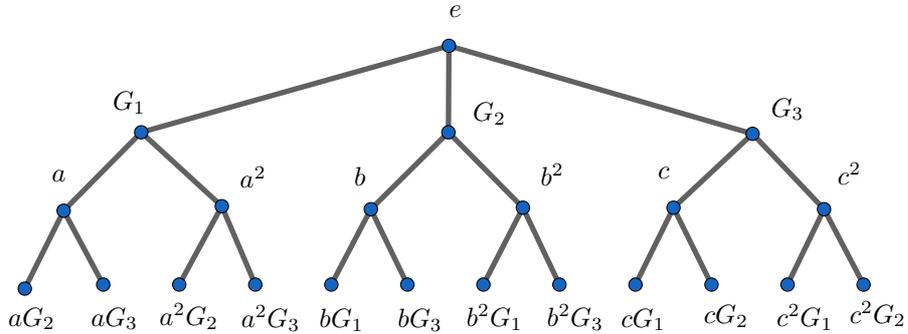
\begin{align*}
\definecolor{wrwrwr}{rgb}{0.3803921568627451,0.3803921568627451,0.3803921568627451}\definecolor{rvwvcq}{rgb}{0.08235294117647059,0.396078431372549,0.7529411764705882}\begin{tikzpicture}[line cap=round,line join=round,>=triangle 45,x=1cm,y=1cm]\clip(-15.507487288888886,3.2663555999999736) rectangle (-2.208253600000002,8.792194355555488);\draw [line width=2pt,color=wrwrwr] (-8.627909955555552,7.872416266666665)-- (-12.674576622222219,6.725749599999998);\draw [line width=2pt,color=wrwrwr] (-8.634576622222221,6.725749599999998)-- (-8.627909955555552,7.872416266666665);\draw [line width=2pt,color=wrwrwr] (-8.634576622222221,6.725749599999998)-- (-9.65457662222222,5.705749599999997);\draw [line width=2pt,color=wrwrwr] (-8.627909955555552,7.872416266666665)-- (-4.634576622222221,6.705749599999997);\draw [line width=2pt,color=wrwrwr] (-12.674576622222219,6.725749599999998)-- (-13.694576622222222,5.685749599999998);\draw [line width=2pt,color=wrwrwr] (-12.674576622222219,6.725749599999998)-- (-11.61457662222222,5.745749599999997);\draw [line width=2pt,color=wrwrwr] (-8.634576622222221,6.725749599999998)-- (-7.654576622222221,5.725749599999998);\draw [line width=2pt,color=wrwrwr] (-4.634576622222221,6.705749599999997)-- (-3.694576622222221,5.705749599999997);\draw [line width=2pt,color=wrwrwr] (-4.634576622222221,6.705749599999997)-- (-5.674576622222221,5.725749599999998);\draw [line width=2pt,color=wrwrwr] (-13.694576622222222,5.685749599999998)-- (-14.205508117249337,4.652314944138265);\draw [line width=2pt,color=wrwrwr] (-13.694576622222222,5.685749599999998)-- (-13.174576622222219,4.705749599999997);\draw [line width=2pt,color=wrwrwr] (-11.61457662222222,5.745749599999997)-- (-12.174576622222219,4.705749599999997);\draw [line width=2pt,color=wrwrwr] (-11.61457662222222,5.745749599999997)-- (-11.174576622222219,4.705749599999997);\draw [line width=2pt,color=wrwrwr] (-9.65457662222222,5.705749599999997)-- (-10.174576622222219,4.705749599999997);\draw [line width=2pt,color=wrwrwr] (-9.65457662222222,5.705749599999997)-- (-9.174576622222222,4.705749599999997);\draw [line width=2pt,color=wrwrwr] (-7.654576622222221,5.725749599999998)-- (-8.174576622222222,4.705749599999997);\draw [line width=2pt,color=wrwrwr] (-7.654576622222221,5.725749599999998)-- (-7.174576622222221,4.705749599999997);\draw [line width=2pt,color=wrwrwr] (-5.674576622222221,5.725749599999998)-- (-6.174576622222221,4.705749599999997);\draw [line width=2pt,color=wrwrwr] (-5.674576622222221,5.725749599999998)-- (-5.174576622222221,4.705749599999997);\draw [line width=2pt,color=wrwrwr] (-3.694576622222221,5.705749599999997)-- (-4.174576622222221,4.705749599999997);\draw [line width=2pt,color=wrwrwr] (-3.694576622222221,5.705749599999997)-- (-3.1745766222222214,4.705749599999997);\draw (-11.499106844444443,6.401481733333283) node[anchor=north west] {$a^2$};\draw (-8.435558933333333,7.246104755555499) node[anchor=north west] {$G_2$};\draw (-13.97571333333333,6.344219155555506) node[anchor=north west] {$a$};\draw (-7.547988977777778,6.430113022222172) node[anchor=north west] {$b^2$};\draw (-14.548339111111108,4.526132311111075) node[anchor=north west] {$aG_2$};\draw (-13.46035013333333,4.526132311111075) node[anchor=north west] {$aG_3$};\draw (-12.55846453333333,4.5833948888888525) node[anchor=north west] {$a^2G_2$};\draw (-11.484791199999998,4.554763599999964) node[anchor=north west] {$a^2G_3$};\draw (-10.4540648,4.511816666666631) node[anchor=north west] {$bG_1$};\draw (-8.521452799999999,4.569079244444408) node[anchor=north west] {$b^2G_1$};\draw (-7.462095111111111,4.569079244444408) node[anchor=north west] {$b^2G_3$};\draw (-6.4886312888888895,4.511816666666631) node[anchor=north west] {$cG_1$};\draw (-5.4006423111111115,4.554763599999964) node[anchor=north west] {$cG_2$};\draw (-9.423338399999999,4.511816666666631) node[anchor=north west] {$bG_3$};\draw (-8.75050311111111,8.53451275555549) node[anchor=north west] {$e$};\draw (-13.174037244444442,7.389261199999942) node[anchor=north west] {$G_1$};\draw (-9.995964177777777,6.387166088888839) node[anchor=north west] {$b$};\draw (-6.001899377777779,6.387166088888839) node[anchor=north west] {$c$};\draw (-4.441494133333334,4.569079244444408) node[anchor=north west] {$c^2G_1$};\draw (-3.482345955555557,4.612026177777741) node[anchor=north west] {$c^2G_2$};\draw (-3.639818044444446,6.444428666666616) node[anchor=north west] {$c^2$};\draw (-4.513072355555557,7.303367333333276) node[anchor=north west] {$G_3$};\begin{scriptsize}\draw [fill=rvwvcq] (-8.627909955555552,7.872416266666665) circle (2.5pt);\draw [fill=rvwvcq] (-12.674576622222219,6.725749599999998) circle (2.5pt);\draw [fill=rvwvcq] (-8.634576622222221,6.725749599999998) circle (2.5pt);\draw [fill=rvwvcq] (-9.65457662222222,5.705749599999997) circle (2.5pt);\draw [fill=rvwvcq] (-4.634576622222221,6.705749599999997) circle (2.5pt);\draw [fill=rvwvcq] (-13.694576622222222,5.685749599999998) circle (2.5pt);\draw [fill=rvwvcq] (-11.61457662222222,5.745749599999997) circle (2.5pt);\draw [fill=rvwvcq] (-7.654576622222221,5.725749599999998) circle (2.5pt);\draw [fill=rvwvcq] (-3.694576622222221,5.705749599999997) circle (2.5pt);\draw [fill=rvwvcq] (-5.674576622222221,5.725749599999998) circle (2.5pt);\draw [fill=rvwvcq] (-14.205508117249337,4.652314944138265) circle (2.5pt);\draw [fill=rvwvcq] (-13.174576622222219,4.705749599999997) circle (2.5pt);\draw [fill=rvwvcq] (-12.174576622222219,4.705749599999997) circle (2.5pt);\draw [fill=rvwvcq] (-11.174576622222219,4.705749599999997) circle (2.5pt);\draw [fill=rvwvcq] (-10.174576622222219,4.705749599999997) circle (2.5pt);\draw [fill=rvwvcq] (-9.174576622222222,4.705749599999997) circle (2.5pt);\draw [fill=rvwvcq] (-8.174576622222222,4.705749599999997) circle (2.5pt);\draw [fill=rvwvcq] (-7.174576622222221,4.705749599999997) circle (2.5pt);\draw [fill=rvwvcq] (-6.174576622222221,4.705749599999997) circle (2.5pt);\draw [fill=rvwvcq] (-5.174576622222221,4.705749599999997) circle (2.5pt);\draw [fill=rvwvcq] (-4.174576622222221,4.705749599999997) circle (2.5pt);\draw [fill=rvwvcq] (-3.1745766222222214,4.705749599999997) circle (2.5pt);\end{scriptsize}\end{tikzpicture}
\end{align*}
Then $f$ moves $G_1$ to the root, dragging every vertex in order to define an isometry:
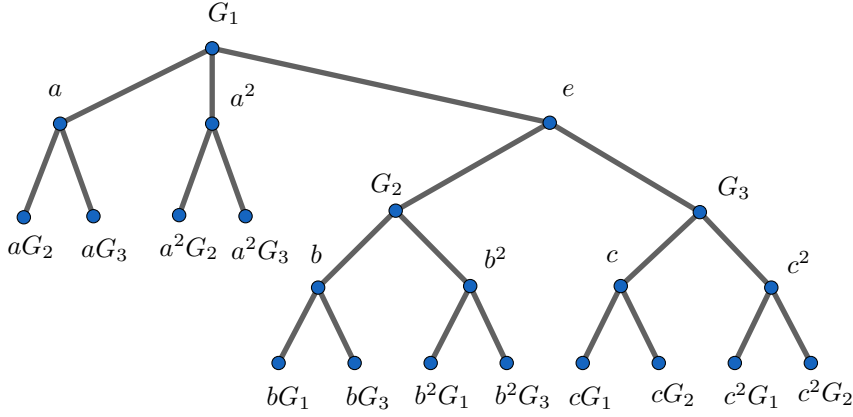
\begin{align*}
\definecolor{wrwrwr}{rgb}{0.3803921568627451,0.3803921568627451,0.3803921568627451}\definecolor{rvwvcq}{rgb}{0.08235294117647059,0.396078431372549,0.7529411764705882}\begin{tikzpicture}[line cap=round,line join=round,>=triangle 45,x=1cm,y=1cm]\clip(-15.576479150149973,2.054671431163178) rectangle (-3.2362627217656645,8.619215786901837);\draw [line width=2pt,color=wrwrwr] (-7.953434563546413,6.659800213009782)-- (-12.392816007667264,7.647212438090747);\draw [line width=2pt,color=wrwrwr] (-9.97948926487423,5.493861372385494)-- (-7.953434563546413,6.659800213009782);\draw [line width=2pt,color=wrwrwr] (-9.97948926487423,5.493861372385494)-- (-10.99948926487423,4.473861372385493);\draw [line width=2pt,color=wrwrwr] (-7.953434563546413,6.659800213009782)-- (-5.97948926487423,5.473861372385493);\draw [line width=2pt,color=wrwrwr] (-12.392816007667264,7.647212438090747)-- (-14.392816007667264,6.647212438090747);\draw [line width=2pt,color=wrwrwr] (-12.392816007667264,7.647212438090747)-- (-12.392816007667264,6.647212438090747);\draw [line width=2pt,color=wrwrwr] (-9.97948926487423,5.493861372385494)-- (-8.99948926487423,4.493861372385494);\draw [line width=2pt,color=wrwrwr] (-5.97948926487423,5.473861372385493)-- (-5.03948926487423,4.473861372385493);\draw [line width=2pt,color=wrwrwr] (-5.97948926487423,5.473861372385493)-- (-7.01948926487423,4.493861372385494);\draw [line width=2pt,color=wrwrwr] (-14.392816007667264,6.647212438090747)-- (-14.870154045665927,5.40605676309403);\draw [line width=2pt,color=wrwrwr] (-14.392816007667264,6.647212438090747)-- (-13.954498717075767,5.420143768149263);\draw [line width=2pt,color=wrwrwr] (-12.392816007667264,6.647212438090747)-- (-12.827538312657108,5.434230773204496);\draw [line width=2pt,color=wrwrwr] (-12.392816007667264,6.647212438090747)-- (-11.954143999232649,5.420143768149263);\draw [line width=2pt,color=wrwrwr] (-10.99948926487423,4.473861372385493)-- (-11.51948926487423,3.473861372385494);\draw [line width=2pt,color=wrwrwr] (-10.99948926487423,4.473861372385493)-- (-10.519489264874233,3.473861372385494);\draw [line width=2pt,color=wrwrwr] (-8.99948926487423,4.493861372385494)-- (-9.519489264874231,3.473861372385494);\draw [line width=2pt,color=wrwrwr] (-8.99948926487423,4.493861372385494)-- (-8.51948926487423,3.473861372385494);\draw [line width=2pt,color=wrwrwr] (-7.01948926487423,4.493861372385494)-- (-7.51948926487423,3.473861372385494);\draw [line width=2pt,color=wrwrwr] (-7.01948926487423,4.493861372385494)-- (-6.51948926487423,3.473861372385494);\draw [line width=2pt,color=wrwrwr] (-5.03948926487423,4.473861372385493)-- (-5.51948926487423,3.473861372385494);\draw [line width=2pt,color=wrwrwr] (-5.03948926487423,4.473861372385493)-- (-4.51948926487423,3.473861372385494);\draw (-12.280119967225398,7.295037311709919) node[anchor=north west] {$a^2$};\draw (-10.434722304989844,6.111728887070332) node[anchor=north west] {$G_2$};\draw (-14.674910826615047,7.295037311709919) node[anchor=north west] {$a$};\draw (-8.941499769135122,5.139725538259243) node[anchor=north west] {$b^2$};\draw (-15.210217018713909,5.2805955888115745) node[anchor=north west] {$aG_2$};\draw (-14.238213669902816,5.238334573645875) node[anchor=north west] {$aG_3$};\draw (-13.20986230087079,5.308769598922041) node[anchor=north west] {$a^2G_2$};\draw (-12.266032962170165,5.266508583756342) node[anchor=north west] {$a^2G_3$};\draw (-11.801161795347468,3.280240870968464) node[anchor=north west] {$bG_1$};\draw (-9.857155097725283,3.336588891189397) node[anchor=north west] {$b^2G_1$};\draw (-8.80062971858279,3.336588891189397) node[anchor=north west] {$b^2G_3$};\draw (-7.828626369771698,3.280240870968464) node[anchor=north west] {$cG_1$};\draw (-6.743926980518738,3.3225018861341638) node[anchor=north west] {$cG_2$};\draw (-10.758723421260209,3.280240870968464) node[anchor=north west] {$bG_3$};\draw (-7.913148400103097,7.295037311709919) node[anchor=north west] {$e$};\draw (-12.575947073385295,8.379736700962873) node[anchor=north west] {$G_1$};\draw (-11.223594588082905,5.167899548369709) node[anchor=north west] {$b$};\draw (-7.335581192838534,5.0833775180383105) node[anchor=north west] {$c$};\draw (-5.786010636762879,3.336588891189397) node[anchor=north west] {$c^2G_1$};\draw (-4.814007287951786,3.378849906355096) node[anchor=north west] {$c^2G_2$};\draw (-4.954877338504119,5.0833775180383105) node[anchor=north west] {$c^2$};\draw (-5.870532667094278,6.069467871904632) node[anchor=north west] {$G_3$};\begin{scriptsize}\draw [fill=rvwvcq] (-7.953434563546413,6.659800213009782) circle (2.5pt);\draw [fill=rvwvcq] (-12.392816007667264,7.647212438090747) circle (2.5pt);\draw [fill=rvwvcq] (-9.97948926487423,5.493861372385494) circle (2.5pt);\draw [fill=rvwvcq] (-10.99948926487423,4.473861372385493) circle (2.5pt);\draw [fill=rvwvcq] (-5.97948926487423,5.473861372385493) circle (2.5pt);\draw [fill=rvwvcq] (-14.392816007667264,6.647212438090747) circle (2.5pt);\draw [fill=rvwvcq] (-12.392816007667264,6.647212438090747) circle (2.5pt);\draw [fill=rvwvcq] (-8.99948926487423,4.493861372385494) circle (2.5pt);\draw [fill=rvwvcq] (-5.03948926487423,4.473861372385493) circle (2.5pt);\draw [fill=rvwvcq] (-7.01948926487423,4.493861372385494) circle (2.5pt);\draw [fill=rvwvcq] (-14.870154045665927,5.40605676309403) circle (2.5pt);\draw [fill=rvwvcq] (-13.954498717075767,5.420143768149263) circle (2.5pt);\draw [fill=rvwvcq] (-12.827538312657108,5.434230773204496) circle (2.5pt);\draw [fill=rvwvcq] (-11.954143999232649,5.420143768149263) circle (2.5pt);\draw [fill=rvwvcq] (-11.51948926487423,3.473861372385494) circle (2.5pt);\draw [fill=rvwvcq] (-10.519489264874233,3.473861372385494) circle (2.5pt);\draw [fill=rvwvcq] (-9.519489264874231,3.473861372385494) circle (2.5pt);\draw [fill=rvwvcq] (-8.51948926487423,3.473861372385494) circle (2.5pt);\draw [fill=rvwvcq] (-7.51948926487423,3.473861372385494) circle (2.5pt);\draw [fill=rvwvcq] (-6.51948926487423,3.473861372385494) circle (2.5pt);\draw [fill=rvwvcq] (-5.51948926487423,3.473861372385494) circle (2.5pt);\draw [fill=rvwvcq] (-4.51948926487423,3.473861372385494) circle (2.5pt);\end{scriptsize}\end{tikzpicture}
\end{align*}
A more precise description is the following: The tree $\mathcal{T}_3$ may be viewed as the Cayley graph of $\Z\ast(\Z/2\Z)$ with generators $-1,1\in\Z$ and $1\in(\Z/2\Z)$. Under this identification, $f$ corresponds to the action by multiplication on the left by $1\in\Z$. Moreover, observe that, for every vertex, $f$ changes its distance to the root by 1. Since $f$ is bijective, this implies that $\Gamma(G)$ is the disjoint union of $\Psi(X)$ and $f(\Psi(X))$. 

Now we will consider products of $\Gamma(G)$, in order to relate them to the products of $X$. The goal is to be able to apply Proposition \ref{Prop_mult_rad} to $\Gamma(G)^N$. From now on, if there is no ambiguity, the letter $d$ will denote the distance in whichever space we are considering.

\begin{prop}\label{prop_nec_prod_Cay}
Let $N\geq 1$ and let $\phi:X^N\times X^N\to\C$ be a radial Schur multiplier with $\phi=\dot{\phi}\circ d$. Then $\dot{\phi}(n)$ converges to some limit $c\in\C$, and the generalised Hankel matrix
\begin{align*}
H= \left(\tbinom{N+i-1}{N-1}^{\frac{1}{2}}\tbinom{N+j-1}{N-1}^{\frac{1}{2}}\mathfrak{d}_1^N\dot{\phi}(i+j)\right)_{i,j\in\N}
\end{align*}
is an element of $S_1(\ell_2(\N))$ of norm at most
\begin{align*}
3^N\left(\|\phi\|_{cb}+|c|\right).
\end{align*}
\end{prop}
\begin{proof}
By restriction, $\phi$ defines a radial Schur multiplier on $X$. Hence, by \cite[Theorem 6.1]{Wys}, $c=\lim_{k\to\infty}\dot{\phi}(k)$ exists. We shall treat first the case when $c=0$. By Theorem \ref{thmHSS}, there is a Hilbert space $\mathcal{H}$ and bounded functions $P,Q:X^N\to\mathcal{H}$ such that
\begin{align*}
\phi(x,y)=\langle P(x),Q(y)\rangle,\quad\forall x,y\in X^N.
\end{align*}
Recall that $\Gamma(G)$ is the disjoint union of $A=\Psi(X)$ and $B=f(\Psi(X))$. Define functions $J:\Gamma(G)^N\to\{0,1\}^N$ and $\psi:\Gamma(G)\to X$ by
\begin{align*}
J(u)_i &=\begin{cases}
0, & \text{if } u_i\in A\\
1, & \text{if } u_i\in B,
\end{cases}\quad \forall i\in\{1,...,N\}.\\
\psi(\omega) &=\begin{cases}
\Psi^{-1}(\omega), & \text{if } \omega\in A\\
\Psi^{-1}\circ f^{-1}(\omega), & \text{if } \omega\in B.
\end{cases}
\end{align*}
Observe that $\Gamma(G)$ is a bipartite graph, and the equality $J(u)=J(v)$ is equivalent to the fact that $d(u_i,v_i)$ is even for all $i\in\{1,...,N\}$. Consider now the Hilbert space $\tilde{\mathcal{H}}=\mathcal{H}\otimes\ell_2\left(\{0,1\}^N\right)$ and define functions $\tilde{P},\tilde{Q}:\Gamma(G)^N\to\tilde{\mathcal{H}}$ by
\begin{align*}
\tilde{P}(u) &=P(\psi(u_1),...,\psi(u_N))\otimes\delta_{J(u)},\\
\tilde{Q}(u) &=Q(\psi(u_1),...,\psi(u_N))\otimes\delta_{J(u)}.
\end{align*}
Observe that $\|\tilde{P}\|_\infty=\| P\|_\infty$ and $\|\tilde{Q}\|_\infty=\| Q\|_\infty$. Moreover, if $u,v\in\Gamma(G)^N$ satisfy $J(u)=J(v)$, then
\begin{align*}
\langle \tilde{P}(u),\tilde{Q}(v)\rangle &= \langle P(\psi(u_1),...,\psi(u_N)),Q(\psi(v_1),...,\psi(v_N))\rangle\\
&= \dot{\phi}(d(\psi(u_1),\psi(v_1))+\cdots + d(\psi(u_N),\psi(v_N)))\\
&= \dot{\phi}\left(\tfrac{1}{2}d(u_1,v_1)+\cdots + \tfrac{1}{2}d(u_N,v_N)\right)\\
&= \dot{\phi}\left(\tfrac{1}{2}d(u,v)\right).
\end{align*}
On the other hand, if $J(u)\neq J(v)$, then $\langle \tilde{P}(u),\tilde{Q}(v)\rangle=0$. We conclude that the function $\varphi:\Gamma(G)^N\times\Gamma(G)^N\to\C$ given by
\begin{align}\label{tildphiuv}
\varphi(u,v)=\langle\tilde{P}(u),\tilde{Q}(v)\rangle,\quad\forall u,v\in \Gamma(G)^N,
\end{align}
is a Schur multiplier on the product of $N$ copies of the 3-homogeneous tree, such that $\|\varphi\|_{cb}\leq\|\phi\|_{cb}$ and
\begin{align*}
\lim_{|n|\to\infty}\varphi(n)=0.
\end{align*}
Moreover, $\varphi$ is a multi-radial function. That is, there exists $\tilde{\varphi}:\N^N\to\C$ such that
\begin{align*}
\varphi(u,v)=\tilde{\varphi}(d(u_1,v_1),...,d(u_N,v_N)),
\end{align*}
for all $u,v\in\Gamma(G)^N$. Namely,
\begin{align*}
\tilde{\varphi}(n)=\begin{cases}
\dot{\phi}\left(\tfrac{1}{2}|n|\right), & \text{if } n\in(2\N)^N\\
0, & \text{otherwise}.
\end{cases}
\end{align*}
Therefore, by Proposition \ref{Prop_mult_rad}, the operator
\begin{align*}
T=\left(\sum_{I\subset[N]}(-1)^{|I|} \tilde{\varphi}(m+n+2\chi^I)\right)_{m,n\in\N^N}
\end{align*}
is an element of $\in S_1(\ell_2(\N^N))$ of norm at most $3^N\|\phi\|_{cb}$. Recall that $\chi^I\in\{0,1\}^N$ is defined as
\begin{align*}
\chi^I_i=\begin{cases} 1 & \text{if } i\in I \\
0 & \text{if } i\notin I.
\end{cases}
\end{align*}
Define now, for each $n\in\N^N$,
\begin{align*}
V\delta_n=\delta_{2n}.
\end{align*}
Then $V$ extends to an isometry on $\ell_2(\N^N)$, which implies that the operator $\tilde{T}=V^*TV$ is an element of $\in S_1(\ell_2(\N^N))$ of norm at most $3^N\|\phi\|_{cb}$. Furthermore, for all $m,n\in\N^N$,
\begin{align*}
\tilde{T}_{m,n} &=T_{2m,2n}\\
&=\sum_{I\subset[N]}(-1)^{|I|} \tilde{\varphi}\left(2(m+n+\chi^I)\right)\\
&=\sum_{I\subset[N]}(-1)^{|I|} \dot{\phi}\left(|m|+|n|+|I|\right)\\
&=\sum_{k=0}^N \tbinom{N}{k} (-1)^k \dot{\phi}\left(|m|+|n|+k\right)\\
&=\mathfrak{d}_1^N\dot{\phi}\left(|m|+|n|\right).
\end{align*}
Hence, by Lemma \ref{Lem_NN_N}, the generalised Hankel matrix
\begin{align*}
H= \left(\tbinom{N+i-1}{N-1}^{\frac{1}{2}}\tbinom{N+j-1}{N-1}^{\frac{1}{2}}\mathfrak{d}_1^N\dot{\phi}(i+j)\right)_{i,j\in\N}
\end{align*}
is an element of $S_1(\ell_2(\N))$ of norm at most $3^N\|\phi\|_{cb}$. This proves the theorem in the case $c=0$. If $c\neq 0$, we repeat the previous argument for $\phi-c$, and since the derivative of a constant function is $0$, we obtain the same conclusion with $\|H\|_{S_1}\leq 3^N\left(\|\phi\|_{cb}+|c|\right)$.
\end{proof}

This completes the proof of Theorem B, since we found a particular case for which the condition $H\in S_1(\ell_2(\N))$ is also necessary.

\section{Sufficient condition for finite dimensional CAT(0) cube complexes}\label{Sec_SCCCC}

In this section we prove Theorem C. We begin by defining median graphs. Let $X$ be (the set of vertices of) a connected graph. For all $x,y\in X$, we define
\begin{align*}
I(x,y)=\{u\in X\, :\, d(x,y)=d(x,u)+d(u,y)\},
\end{align*}
where $d:X\times X\to\N$ is the combinatorial distance. Observe that $I(x,y)$ is the union of all the vertices lying in a geodesic joining $x$ to $y$. We call it the interval between $x$ and $y$.  Now define
\begin{align*}
I(x,y,z)=I(x,y)\cap I(y,z)\cap I(z,x),
\end{align*}
for all $x,y,z\in X$. We say that $X$ is a median graph if
\begin{align*}
|I(x,y,z)|=1,\quad\forall x,y,z\in X.
\end{align*}
In that case, we call median of $x,y,z$, the unique element $\mu(x,y,z)\in I(x,y,z)$. It is not hard to check that trees, and more generally, products of trees are median. The following theorem of Chepoi relates median graphs to CAT(0) cube complexes. For details on CAT(0) cube complexes, see \cite[\S 2]{GueHig}.

\begin{thm}{\cite[Theorem 6.1]{Che}}
Median graphs are exactly the 1-skeletons of CAT(0) cube complexes.
\end{thm}

This powerful result allows us to analyse these objects from two different points of view. We shall focus mainly on the median graph structure; however, the notion hyperplane will be useful. A hyperplane in a CAT(0) cube complex is an equivalence class of edges under the equivalence relation generated by
\begin{align*}
\{x,y\}\sim \{u,v\}\quad\text{if}\quad\{x,y,v,u\}\text{ is a square}.
\end{align*}
If $H$ is a hyperplane and $\{x,y\}\in H$, we say that $H$ separates $x$ from $y$, and that $\{x,y\}$ crosses $H$. More generally, we say that a path crosses $H$ if one of its edges does.

\begin{thm}{\cite[Theorem 4.13]{Sag}}\label{thmSag}
Let $x,y$ be two vertices in a CAT(0) cube complex $X$, and let $\gamma$ be a geodesic joining $x$ and $y$. Then $\gamma$ crosses every hyperplane separating $x$ from $y$, and it does so only once. Moreover, $\gamma$ does not cross any other hyperplane.
\end{thm}

Mizuta \cite{Miz} proved that groups acting properly on finite dimensional CAT(0) cube complexes are weakly amenable by making use of their median structure. We quickly describe his construction, as it will be our main tool to prove Theorem C. Let $X$ be a median graph of dimension $N<\infty$ (as a cube complex) and $\mu:X^3\to X$ its median function. Like we did in the case of trees, we fix an infinite geodesic $\omega_o:\N\to X$. Observe that we can always assume that such a geodesic exists since the fact of adding an infinite ray starting from a vertex of the complex preserves both the fact of being median and the dimension.

\begin{lem}{\cite[Lemma 2]{Miz}}\label{lemmedian}
Let $x_1,x_2\in X$. There exists a unique point $m(x_1,x_2)\in X$ such that, for all but finitely many $z\in\omega_o$, $\mu(x_1, x_2, z)=m(x_1,x_2)$.
\end{lem}

For $x\in X$ and $k\in\N$, we put
\begin{align*}
A(x,k)=\left\{y\in X\ :\ \begin{matrix}
\exists\,\omega_x \text{ infinite geodesic such that}\\
\omega_x(0)=x,\, \omega_x(k)=y,\, |\omega_x\Delta\omega_0|<\infty
\end{matrix}
\right\}.
\end{align*}
Observe that, for a tree, $A(x,k)=\{\omega_x(k)\}$, where $\omega_x$ is the unique geodesic that satisfies $\omega_x(0)=x$ and $|\omega_x\Delta\omega_0|<\infty$. The sets $A(x,k)$ can be endowed with a polytopal structure as follows. We define $0$-polytopes as the set of vertices of $X$. For $l\in\{1,...,N-1\}$, we say that $P\subset X$ is an $l$-polytope if there exists an $(l+1)$-cube $C$, a vertex $w\in C$ and $j\in\{1,...,l\}$ such that $P$ is the set of points at distance $j$ from $w$, lying in a geodesic between $w$ and $d_C(w)$, where $d_C(w)$ is the point diagonal to $w$ with respect to $C$. For $l\in\{0,...,N-1\}$, we define $\mathcal{A}(x,k)^{(l)}$ as the set of $l$-polytopes contained in $A(x,k)$. Hence, the set of all polytopes in $A(x,k)$ is
\begin{align*}
\mathcal{A}(x,k)=\bigcup_{l=0}^{N-1}\mathcal{A}(x,k)^{(l)}.
\end{align*}
These sets of polytopes will play the role of the delta functions $\delta_{\omega_x}$ in the proof of Lemma \ref{Lem_suf_mult_rad}. More precisely, for $k\in\N$ and $l\in\{0,...,N-1\}$, define maps $f_k^{(l)}:X\to\bigoplus_{j=0}^{N-1}\ell_2(\mathcal{X}^{(j)})$, where $\mathcal{X}^{(j)}$ is the set of all $j$-polytopes in $X$, by
\begin{align*}
f_k^{(l)}(x)=\sum_{P\in \mathcal{A}(x,k)^{(l)}}\delta_P \in \ell_2(\mathcal{X}^{(l)})\subseteq \bigoplus_{j=0}^{N-1}\ell_2(\mathcal{X}^{(j)}).
\end{align*}
It follows from the definition that $\|f_k^{(l)}(x)\|^2=|\mathcal{A}(x,k)^{(l)}|$. Now put
\begin{align}\label{PkQk}
P_k(x)=\sum_{l=0}^{N-1}f_k^{(l)}(x)\quad\text{and}\quad Q_k(x)=\sum_{l=0}^{N-1}(-1)^{l}f_k^{(l)}(x).
\end{align}
Observe that, if $k\neq j$, then $\langle P_k(x),P_j(x)\rangle=0$ and $\langle Q_k(x),Q_j(x)\rangle=0$. The following result is implicit in the proof of \cite[Theorem 2]{Miz}.

\begin{prop}\label{prodPkQk}
Let $x_1,x_2\in X$ and $k_1,k_2\in\N$. Then
\begin{align*}
\langle P_{k_1}(x_1),Q_{k_2}(x_2)\rangle=\left\{\begin{array}{ll}
1 & \text{if } \exists j\geq 0,\ k_i=l_i+j\ (i=1,2)\\
0 & \text{if not}
\end{array}\right. ,
\end{align*}
where $l_i=d(x_i,m(x_1,x_2))$ ($i=1,2$) and $m(x_1,x_2)$ is the point given by Lemma \ref{lemmedian}.
\end{prop}

Moreover, from the definition we have
\begin{align*}
\|P_k(x)\|^2=\sum_{l=0}^{N-1}\|f_k^{(l)}(x)\|^2=\sum_{l=0}^{N-1}|\mathcal{A}(x,k)^{(l)}|=|\mathcal{A}(x,k)|.
\end{align*}
And the same holds for $\|Q_k(x)\|^2$. The following lemma says that $|A(x,k)|=O(k^{N-1})$. We will prove that this is also true for $|\mathcal{A}(x,k)|$, and therefore, for $\|P_k(x)\|^2$ and $\|Q_k(x)\|^2$.

\begin{lem}{\cite[Lemma 5]{Miz}}\label{cardAxk}
For all $x\in X$ and $k\in\N$,
\begin{align*}
|A(x,k)|\leq\tbinom{N-1+k}{N-1}.
\end{align*}
\end{lem}

Observe that, if we assume that the degrees of the points in $X$ are uniformly bounded, we can conclude that there exists a constant $M>0$ such that every vertex belongs to at most $M$ polytopes. In what follows, we show that this is always true when we restrict ourselves to the polytopes in $\mathcal{A}(x,k)$. For $x\in X$, $k\in\N$, $y\in A(x,k)$ and $i=0,...,\min\{N-1,k\}$ define
\begin{align*}
B_i(x,y)=\{w\in A(x,k-i)\ :\ y\in A(w,i)\}.
\end{align*}

\begin{lem}\label{cardBi}
For all $x\in X$, $k\in\N$, $y\in A(x,k)$ and $i=0,...,\min\{N-1,k\}$, we have
\begin{align*}
|B_i(x,y)|\leq N^i.
\end{align*}
\end{lem}
\begin{proof}
We proceed by induction on $i$. Observe first that, for every $x$, $k$ and $y$ as above, $B_0(x,y)=\{y\}$, hence the result holds with equality for $i=0$. If $k\geq 1$, for every $z\in B_1(x,y)$, let $H_z$ be the hyperplane separating $z$ from $y$. Observe that, by Theorem \ref{thmSag}, $H_z$ also separates $y$ from $x$, since $z$ lies in a geodesic joining $x$ and $y$. Then \cite[Proposition 2.8]{GueHig} implies that there is a cube $C$ where all the hyperplanes $H_z$ intersect. Hence $B_1(x,y)\cup\{y\}$ is included in $C$, which is of dimension at most $N$. Therefore $|B_1(x,y)|\leq N$. Finally assume that the result holds for some $i<N-1$, and take $x\in X$, $k\geq i+1$, $y\in A(x,k)$. Then
\begin{align*}
B_{i+1}(x,y)=\bigcup_{z\in B_1(x,y)}B_i(x,z).
\end{align*}
Thus 
\begin{align*}
|B_{i+1}(x,y)|\leq  |B_1(x,y)| \left(\sup_{z\in B_1(x,y)}|B_i(x,z)|\right)\leq N N^i=N^{i+1}.
\end{align*}
\end{proof}

\begin{lem}\label{cardpoly}
There is a constant $M>0$, depending only on the dimension $N$, such that for every $x\in X$, $k\in\N$ and $y\in A(x,k)$,
\begin{align*}
\left| \{P\in\mathcal{A}(x,k)\ :\ y\in P\}\right|\leq M.
\end{align*}
\end{lem}
\begin{proof}
Fix $x$, $k$ and $y$ as above and put $\mathcal{Q}=\{P\in\mathcal{A}(x,k)\ :\ y\in P\}$. If $P\in\mathcal{Q}$, then $P$ is an $l$-polytope ($0\leq l\leq N-1$) and by definition there exist an $(l+1)$-cube $C$ and $z\in C$ such that $P$ is the subset of $C$ consisting of elements at a fixed distance from $z$. Let $\tilde{z}\in C$ be the point diagonal to $z$ with respect to $C$. Without loss of generality, we may assume that $d(x,z)\leq d(x,\tilde{z})$. By \cite[Lemma 4]{Miz}, $P\subseteq A(z,i)$ and $z\in A(x,k-i)$, where $i=d(z,P)\in\{0,...,\min\{k,N-1\}\}$. This means that $z\in B_i(x,y)$. So, if we put
\begin{align*}
A=\bigcup_{i=0}^{\min\{k,N-1\}}\bigcup_{z\in B_i(x,y)}A(z,i),
\end{align*}
then $P\subseteq A$. Using Lemmas \ref{cardAxk} and \ref{cardBi}, we get
\begin{align*}
|A| &\leq\sum_{i=0}^{\min\{k,N-1\}}\sum_{z\in B_i(x,y)}|A(z,i)| \\
&\leq \sum_{i=0}^{\min\{k,N-1\}} |B_i(x,y)| \tbinom{N-1+i}{N-1} \\
&\leq \sum_{i=0}^{N-1} N^i \tbinom{N-1+i}{N-1},
\end{align*}
and this depends only on $N$. Since $\mathcal{Q}\subseteq\mathcal{P}(A)$, we have $|\mathcal{Q}|\leq 2^{|A|}$, and the result follows.
\end{proof}

\begin{lem}\label{boundPk}
Let $N\geq 1$ and let $M$ be the constant in Lemma \ref{cardpoly}. Then, for all $x\in X$ and $k\in\N$,
\begin{align*}
|\mathcal{A}(x,k)|\leq M\tbinom{N-1+k}{N-1}.
\end{align*}
Hence, the same holds for $\|P_k(x)\|^2$ and $\|Q_k(x)\|^2$.
\end{lem}
\begin{proof}
Let $x\in X$ and $k\in\N$. Observe that
\begin{align*}
\mathcal{A}(x,k)=\bigcup_{y\in A(x,k)}\{P\in\mathcal{A}(x,k)\ :\ y\in P\}. 
\end{align*}
Thus, by Lemmas \ref{cardpoly} and \ref{cardAxk},
\begin{align*}
|\mathcal{A}(x,k)|\leq M |A(x,k)| \leq M \tbinom{N-1+k}{N-1}.
\end{align*}
\end{proof}

\begin{proof}[Proof of Theorem C]
The proof follows the same idea as that of Lemma \ref{Lem_suf_mult_rad}, by replacing the delta functions $\delta_{\omega_x(k)}$ by the vectors $P_k(x)$ and $Q_k(x)$ defined in \eqref{PkQk}. Let $D\in\mathcal{B}(\ell_2(\N))$ be the diagonal matrix defined by
\begin{align*}
D_{i,i}=\tbinom{N+i-1}{N-1}^{-\frac{1}{2}},
\end{align*}
and put $\tilde{H}=DHD=(\mathfrak{d}_2\dot{\phi}(i+j))_{i,j\geq 0}$. Then $\tilde{H}\in S_1(\ell_2(\N))$, and by Theorem \ref{thmHSS}, we know that the limits $\lim_{n\to\infty}\dot{\phi}(2n), \lim_{n\to\infty}\dot{\phi}(2n+1)$ exist. Now take $A,B\in S_2(\ell_2(\N))$ such that $H=A^*B$  and $\|H\|_{S_1}=\|A\|_{S_2}\|B\|_{S_2}$. We have $\tilde{H}=\tilde{A}^*\tilde{B}$, where $\tilde{A}=AD\in S_2(\ell_2(\N))$ and $\tilde{B}=BD\in S_2(\ell_2(\N))$.  Define functions $P$ and $Q$ by
\begin{align*}
P(x)=\sum_{k\geq 0}P_k(x)\otimes \tilde{B}e_k,\quad Q(x)=\sum_{k\geq 0}Q_k(x)\otimes \tilde{A}e_k,\quad \forall x\in X,
\end{align*}
where $P_k$ and $Q_k$ are as in (\ref{PkQk}), and $\{e_k\}_{k\in\N}$ is the canonical orthonormal basis of $\ell_2(\N)$. By Proposition \ref{prodPkQk}, we have
\begin{align*}
\langle P_{k_1}(x),Q_{k_2}(y)\rangle=\left\{\begin{array}{ll}
1 & \text{if } \exists j\geq 0,\ k_i=l_i+j\ (i=1,2)\\
0 & \text{if not}
\end{array}\right. ,
\end{align*}
where $l_1=d(x,m(x,y))$ and $l_2=d(y,m(x,y))$. This implies in particular that $l_1+l_2=d(x,y)$. We obtain
\begin{align*}
\langle P(x),Q(y)\rangle &= \sum_{j\geq 0}\langle\tilde{A}^*\tilde{B}e_{l_1+j},e_{l_2+j}\rangle \\
&= \sum_{j\geq 0}\tilde{H}_{l_1+j,l_2+j} \\
&= \sum_{j\geq 0}\mathfrak{d}_2\dot{\phi}(d(x,y)+2j)\\
&= \sum_{j\geq 0}\dot{\phi}(d(x,y)+2j)-\dot{\phi}(d(x,y)+2(j+1))\\
&= \dot{\phi}(d(x,y)) - \lim_{n\to\infty} \dot{\phi}(d(x,y)+2n)\\
&= \dot{\phi}(d(x,y)) - \left(c_+ + c_-(-1)^{d(x,y)}\right).
\end{align*}
Moreover, by Lemma \ref{boundPk}, we have
\begin{align*}
\|P(x)\|^2 &= \sum_{k\geq 0}\|P_k(x)\|^2\|\tilde{B}e_k\|^2\\
&= \sum_{k\geq 0}|\mathcal{A}(x,k)|\|BDe_k\|^2\\
&\leq M\sum_{k\geq 0}\tbinom{N-1+k}{N-1}\tbinom{N-1+k}{N-1}^{-1}\|Be_k\|^2\\
&= M\|B\|_{S_2}^2.
\end{align*}
Similar computations show that $\|Q(y)\|^2\leq M\|A\|_{S_2}^2$, and these bounds do not depend on $x$ or $y$. This, together with Lemma \ref{lemoddeven}, implies that $\phi$ is a Schur multiplier and
\begin{align*}
\|\phi\|_{cb}\leq M \|H\|_{S_1} + |c_+| + |c_-|.
\end{align*}
\end{proof}

\section{Inclusions of sets of multipliers}\label{Sect_incl}

In this section, we show how the conditions in Theorems A, B and C relate to each other. We begin by giving a quick introduction to Besov spaces on the torus. For a more detailed treatment, we refer the reader to \cite[Appendix 2.6]{Pel}. For $n\geq 1$, let $W_n:\T\to\C$ be the polynomial whose Fourier coefficients are given by
\begin{align*}
\hat{W}_n(k)=\begin{cases}
2^{-n+1}(k-2^{n-1}),& \text{if } k\in[2^{n-1},2^n]\\
2^{-n}(2^{n+1}-k),& \text{if } k\in[2^{n},2^{n+1}]\\
0,& \text{otherwise},
\end{cases}
\end{align*}
and put $W_0(z)=1+z$. For $s\in\R$, we define the Besov space of analytic functions $B_1^s(\T)$ as the space of all series $\varphi(z)=\sum_{n\geq 0}a_nz^n$ such that
\begin{align*}
\|\varphi\|_{B_1^s}=\sum_{n\geq 0}2^{ns}\|W_n\ast\varphi\|_{L_1(\T)}<\infty.
\end{align*}
Observe that the notation is not the same as that of \cite{Pel}. Since we only deal with the particular case of analytic functions, the space $B_1^s(\T)$ corresponds to $(B_1^s)_+$ in \cite{Pel}. The following result relates Besov spaces to generalised Hankel matrices in $S_1(\ell_2(\N))$. It is a particular case of \cite[Theorem 6.8.9]{Pel}.

\begin{thm}[Peller]\label{thmPeller}
Let $\varphi:\D\to\C$ be an analytic function, and $\alpha,\beta>-\frac{1}{2}$. Then the generalised Hankel matrix 
\begin{align*}
((1+i)^\alpha(1+j)^\beta\hat{\varphi}(i+j))_{i,j\geq 0}
\end{align*}
defines an element of $S_1(\ell_2(\N))$ if and only if $\varphi\in B_{1}^{1+\alpha+\beta}(\T)$.
\end{thm}

This result allows us to show that Besov spaces are invariant under the shift operator.

\begin{prop}\label{Prop_Bshift}
Let $s>0$ and $\sum_{n\geq 0}a_nz^n\in B_{1}^{s}(\T)$. Then the series 
\begin{align*}
\sum_{n\geq 0}a_{n+1}z^n\quad\text{and}\quad \sum_{n\geq 1}a_{n-1}z^n
\end{align*}
define elements of $B_{1}^{s}(\T)$ as well.
\end{prop}
\begin{proof}
Take $\alpha=\frac{s-1}{2}$. By Theorem \ref{thmPeller}, the matrix $H_{i,j}=(1+i)^\alpha(1+j)^\alpha a_{i+j}$ belongs to $S_1(\ell_2(\N))$. Let $S$ be the forward shift operator on $\ell_2(\N)$ and let $D\in\mathcal{B}(\ell_2(\N))$ be the diagonal operator given by $D_{i,i}=\left(\frac{1+i}{2+i}\right)^{\alpha}$. Then $\tilde{H}=DS^*H$ belongs to $S_1(\ell_2(\N))$, and
\begin{align*}
\tilde{H}_{i,j}=\langle H\delta_j,SD\delta_i\rangle = \left(\frac{1+i}{2+i}\right)^{\alpha}\langle H\delta_j,\delta_{i+1}\rangle = (1+i)^\alpha(1+j)^\alpha a_{i+j+1}.
\end{align*}
Again by Theorem \ref{thmPeller}, we conclude that $\sum_{n\geq 0}a_{n+1}z^n\in B_{1}^{s}(\T)$.\\
Now observe that $D^{-1}$ is a bounded operator as well, so $H'=SD^{-1}H\in S_1(\ell_2(\N))$ and
\begin{align*}
H_{i,j}'=\begin{cases}
0, & \text{if } i=0,\\
(1+i)^{\alpha}(1+j)^{\alpha}a_{i+j-1},& \text{if } i>0.
\end{cases}
\end{align*}
Consider also the rank-1 operator $R$ defined by
\begin{align*}
R_{i,j}=\begin{cases}
(1+j)^{\alpha}a_{j-1}, & \text{if } i=0, j>0,\\
0, & \text{otherwise},
\end{cases}
\end{align*}
and observe that
\begin{align*}
\|R\|_{S_1}=\left(\sum_{j\geq 1}(1+j)^{2\alpha}|a_{j-1}|^2\right)^{\frac{1}{2}}=\|H'e_0\|_2 <\infty.
\end{align*}
Hence $H'+R\in S_1(\ell_2(\N))$, and (putting $a_{-1}=0$),
\begin{align*}
H_{i,j}'+R_{i,j}=(1+i)^{\alpha}(1+j)^{\alpha}a_{i+j-1}.
\end{align*}
By Theorem \ref{thmPeller}, we have $\sum_{n\geq 1}a_{n-1}z^n\in B_{1}^{s}(\T)$.
\end{proof}

We will also need to make use of the operators $I_\alpha$ of fractional integration. For all $\alpha,s\in\R$, and $\varphi\in B_1^s$, define $I_\alpha\varphi$ by
\begin{align}\label{fract_int}
I_{\alpha}\varphi (z)=\sum_{n\geq 0}(1+n)^{-\alpha}\hat{\varphi}(n) z^n.
\end{align}
This operator satisfies $I_\alpha B_1^s = B_1^{s+\alpha}$. For a proof, see e.g. \cite[Theorem 2.4]{dlS}, where this is done in the more general context of vector-valued Besov spaces.

Let $(a_n)_{n\in\N}$ and $(b_n)_{n\in\N}$ be two sequences of complex numbers. We will write $a_n\sim b_n$ if there exist $C_1,C_2>0$ such that
\begin{align*}
C_1|b_n|\leq|a_n|\leq C_2|b_n|,\quad\forall n\in\N.
\end{align*}

\begin{lem}\label{lemH123}
Let $(a_n)_{n\in\N}$ be a sequence of complex numbers and $\alpha,\beta>-\frac{1}{2}$. Consider the matrices
\begin{align*}
H_1&= ((1+i)^\alpha(1+j)^\beta a_{i+j})_{i,j\in\N},\\
H_2&= ((1+i+j)^{\alpha+\beta}a_{i+j})_{i,j\in\N}.
\end{align*}
Then $H_1\in S_1(\ell_2(\N))$ if and only if $H_2\in S_1(\ell_2(\N))$. Moreover, when $\alpha+\beta\in\N$, these conditions are equivalent to
\begin{align}\label{genHankH3}
H_3=\left(\tbinom{\alpha+\beta+i}{\alpha+\beta}^{\frac{1}{2}}\tbinom{\alpha+\beta+j}{\alpha+\beta}^{\frac{1}{2}}a_{i+j}\right)_{i,j\in\N}\in S_1(\ell_2(\N)).
\end{align}
\end{lem}
\begin{proof}
By Theorem \ref{thmPeller}, $H_1$ belongs to $S_1(\ell_2(\N))$ if and only if the analytic function $\varphi(z)=\sum_{n\geq 0}a_nz^n$ belongs to the Besov space $B_1^{1+\alpha+\beta}(\T)$. This is equivalent to the fact that the analytic function
\begin{align*}
I_{-\alpha-\beta}\varphi (z)=\sum_{n\geq 0}(1+n)^{\alpha+\beta}a_nz^n
\end{align*}
belongs to $B_1^1(\T)$, where $I_{-\alpha-\beta}$ is defined as in \eqref{fract_int}. And again by Theorem \ref{thmPeller}, this is equivalent to $H_2\in S_1(\ell_2(\N))$. Now suppose that $m=\alpha+\beta$ is a natural number. Then
\begin{align*}
\tbinom{m+k}{m}^{\frac{1}{2}}\sim (1+k)^\frac{m}{2},
\end{align*}
which implies that the diagonal matrix $D$ given by $D_{ii}=\binom{m+i}{m}^{-\frac{1}{2}}(1+i)^\frac{m}{2}$ defines a bounded operator on $\ell_2(\N)$, and so does $D^{-1}$. Put $\tilde{H}_1=((1+i)^\frac{m}{2}(1+j)^\frac{m}{2} a_{i+j})_{i,j\in\N}$. Observing that $\tilde{H}_1=DH_3D$ and $\alpha+\beta=\frac{m}{2}+\frac{m}{2}$, we get
\begin{align*}
H_3\in S_1(\ell_2(\N)) &\iff \tilde{H}_1\in S_1(\ell_2(\N))\\
&\iff H_2\in S_1(\ell_2(\N)).
\end{align*}
\end{proof}

Let us now introduce some notation. For each $N\geq 1$ and $\dot{\phi}:\N\to\C$, define the following infinite matrices,
\begin{align*}
A(N,\dot{\phi})&= \left((1+i+j)^{N-1}\mathfrak{d}_2^N\dot{\phi}(i+j)\right)_{i,j\in\N}\\
B(N,\dot{\phi})&= \left((1+i+j)^{N-1}\mathfrak{d}_1^N\dot{\phi}(i+j)\right)_{i,j\in\N}\\
C(N,\dot{\phi})&= \left((1+i+j)^{N-1}\mathfrak{d}_2\dot{\phi}(i+j)\right)_{i,j\in\N},
\end{align*}
and consider the following subspaces of $\ell_\infty(\N)$,
\begin{align*}
\mathcal{A}_N&=\left\{\dot{\phi}\in\ell_\infty(\N)\, :\, A(N,\dot{\phi})\in S_1(\ell_2(\N)) \right\}\\
\mathcal{B}_N&=\left\{\dot{\phi}\in\ell_\infty(\N)\, :\, B(N,\dot{\phi})\in S_1(\ell_2(\N)) \right\}\\
\mathcal{C}_N&=\left\{\dot{\phi}\in\ell_\infty(\N)\, :\, C(N,\dot{\phi})\in S_1(\ell_2(\N)) \right\}.
\end{align*}

In other words, thanks to Lemma \ref{lemH123}, a function $\dot{\phi}:\N\to\C$ belongs to $\mathcal{A}_N$ (resp. $\mathcal{B}_N$, $\mathcal{C}_N$) if and only if it satisfies the condition in Theorem A (resp. B, C). Moreover, this can be stated in terms of Besov spaces.

\begin{lem}\label{Lem_carBesov}
Let $N\geq 1$ and $\dot{\phi}:\N\to\C$ be a bounded function. Then $\dot{\phi}$ belongs to $\mathcal{A}_N$ (resp. $\mathcal{B}_N$, $\mathcal{C}_N$) if and only if the analytic function
\begin{align*}
(z^2-1)^N\sum_{n\geq 0}\dot{\phi}(n)z^n \quad\left(\text{resp. } (z-1)^N\sum_{n\geq 0}\dot{\phi}(n)z^n,\quad (z^2-1)\sum_{n\geq 0}\dot{\phi}(n)z^n\right)
\end{align*}
belongs to $B_1^N(\T)$.
\end{lem}
\begin{proof}
We shall only prove the first case since the other ones are analogous. Extend $\dot{\phi}$ to $\Z$ by setting $\dot{\phi}(n)=0$ for $n<0$, and observe that
\begin{align*}
(z^2-1)^N\sum_{n\geq 0}\dot{\phi}(n)z^n &= \sum_{k=0}^N\sum_{n\geq 0}\tbinom{N}{k}(-1)^k\dot{\phi}(n)z^{n+2N-2k}\\
&= \sum_{k=0}^N\sum_{n\geq 0}\tbinom{N}{k}(-1)^k\dot{\phi}(n-2N+2k)z^{n}\\
&= \sum_{n\geq 0}\mathfrak{d}_2^N\dot{\phi}(n-2N)z^{n}.
\end{align*}
This, together with the fact that polynomials belong to $B_1^N(\T)$, implies that $(z^2-1)^N\sum_{n\geq 0}\dot{\phi}(n)z^n$ belongs to $B_1^N(\T)$ if and only if
\begin{align*}
\sum_{n\geq 2N}\mathfrak{d}_2^N\dot{\phi}(n-2N)z^{n}
\end{align*}
does. By Proposition \ref{Prop_Bshift}, this is equivalent to 
\begin{align*}
\sum_{n\geq 0}\mathfrak{d}_2^N\dot{\phi}(n)z^{n} \in B_1^N(\T),
\end{align*}
which, by Theorem \ref{thmPeller}, is equivalent to the fact that $\dot{\phi}\in \mathcal{A}_N$.
\end{proof}

The main goal in this section is to prove the following.
\begin{prop}\label{prop_inclusions}
The sets $\mathcal{A}_N$, $\mathcal{B}_N$ and $\mathcal{C}_N$ satisfy the following relations.
\begin{itemize}
\item[a)] For all $N\geq 1$, $\mathcal{A}_{N+1}\subsetneq \mathcal{A}_N$, $\mathcal{B}_{N+1}\subsetneq \mathcal{B}_N$ and $\mathcal{C}_{N+1}\subsetneq \mathcal{C}_N$.
\item[b)] For all $N\geq 2$, $\mathcal{C}_N\subsetneq \mathcal{A}_N$ and $\displaystyle\bigcap_{m\geq 1}\mathcal{A}_m\nsubseteq \mathcal{C}_N$.
\item[c)] For all $N\geq 1$, $\mathcal{B}_N\subsetneq \mathcal{A}_N$ and $\displaystyle\bigcap_{m\geq 1}\mathcal{A}_m\nsubseteq \mathcal{B}_N$. Furthermore, $\displaystyle\bigcap_{m\geq 1}\mathcal{C}_m\nsubseteq \mathcal{B}_N$.
\end{itemize}
\end{prop}

We will concentrate first in the strict inclusion $\mathcal{B}_{N+1}\subsetneq \mathcal{B}_N$ by studying the function
\begin{align}\label{goodphi(-1)n}
\dot{\phi}(n)=\frac{(-1)^n}{(n+1)^{\alpha+1}},
\end{align}
for different values of $\alpha\geq 0$. The inclusion $\mathcal{A}_{N+1}\subsetneq \mathcal{A}_N$ follows by the same arguments, considering the function $n\mapsto\frac{\ii^n}{(n+1)^{\alpha+1}}$ instead. The verifications will be left to the reader. 

\begin{lem}\label{lemc_n_sim}
Let $\alpha>0$ and define
\begin{align*}
a_n=\frac{(-1)^n}{(n+1)^\alpha},\quad\forall n\in\N.
\end{align*}
Then $(\mathfrak{d}_1^ma)_n\sim \frac{1}{(n+1)^\alpha}$ for all $m\in\N$.
\end{lem}
\begin{proof}
By Lemma \ref{lemsumbinom},
\begin{align*}
(\mathfrak{d}_1^ma)_n &=\sum_{k=0}^m\tbinom{m}{k}(-1)^k  \frac{(-1)^{n+k}}{(n+k+1)^\alpha}\\
&=(-1)^n\sum_{k=0}^m\tbinom{m}{k} \frac{1}{(n+k+1)^\alpha}\\
&\sim \frac{1}{(n+1)^\alpha}.
\end{align*}
\end{proof}

Observe that the sets $\mathcal{A}_N$, $\mathcal{B}_N$ and $\mathcal{C}_N$ are defined in terms of Hankel matrices, but in general it is not easy to determine if a sequence $(a_n)$ defines a Hankel matrix in $S_1(\ell_2(\N))$. The following theorem of Bonsall \cite{Bon} provides a sufficient condition.

\begin{thm}{\cite[Theorem 3.1]{Bon}}\label{thmBonsall}
Let $(a_n)$ be a sequence of complex numbers converging to $0$ and such that
\begin{align*}
\sum_{n\geq 2} |a_{n-1}-a_n| n\log n <\infty.
\end{align*}
Then the Hankel matrix $(a_{i+j})_{i,j\in\N}$ belongs to $S_1(\ell_2(\N))$.
\end{thm}

\begin{lem}\label{lem_phi_alpha}
Let $\alpha\geq 0$ and $m\in\N$. Define $\dot{\phi}:\N\to\C$ by
\begin{align*}
\dot{\phi}(n)=\frac{(-1)^n}{(n+1)^{\alpha+1}}.
\end{align*}
Then
\begin{align*}
((1+i+j)^{s}\mathfrak{d}_1^m\dot{\phi}(i+j))_{i,j\in\N}\notin &S_1(\ell_2(\N)),\quad\forall s\geq\alpha\\
((1+i+j)^{s}\mathfrak{d}_1^m\dot{\phi}(i+j))_{i,j\in\N}\in &S_1(\ell_2(\N)),\quad\forall s<\alpha-1.
\end{align*}
\end{lem}
\begin{proof}
Consider first $s\geq\alpha$ and observe that the series $\sum_{i\geq 0}(2i+1)^{s}\mathfrak{d}_1^m\dot{\phi}(2i)$ diverges because $\mathfrak{d}_1^m\dot{\phi}(n)\sim (n+1)^{-\alpha-1}$, by Lemma \ref{lemc_n_sim}. This implies that
\begin{align*}
((1+i+j)^{s}\mathfrak{d}_1^m\dot{\phi}(i+j))_{i,j\in\N}\notin S_1(\ell_2(\N)),
\end{align*}
since otherwise its diagonal would be an element of $\ell_1$. Now suppose that $s<\alpha-1$. Then, again by Lemma \ref{lemc_n_sim},
\begin{align*}
\left|n^s\mathfrak{d}_1^m\dot{\phi}(n-1)-(n+1)^s\mathfrak{d}_1^m\dot{\phi}(n)\right| 
&\leq n^s\left|\mathfrak{d}_1^m\dot{\phi}(n-1)\right| + (n+1)^s\left|\mathfrak{d}_1^m\dot{\phi}(n)\right|\\
&\sim n^s n^{-\alpha-1}.
\end{align*}
Hence
\begin{align*}
\sum_{n\geq 2} \left|n^s\mathfrak{d}_1^m\dot{\phi}(n-1)-(n+1)^s\mathfrak{d}_1^m\dot{\phi}(n)\right| n\log n \leq C\sum_{n\geq 2} n^{s-\alpha}\log n <\infty.
\end{align*}
By Theorem \ref{thmBonsall}, $((1+i+j)^{s}\mathfrak{d}_1^m\dot{\phi}(i+j))_{i,j\in\N}\in S_1(\ell_2(\N))$.
\end{proof}

The following is a direct consequence of Lemma \ref{lem_phi_alpha}.

\begin{cor}\label{corB_N-B_N+2}
Let $N\geq 1$ and $\alpha\in(N,N+1]$. Then the function $\dot{\phi}:\N\to\C$ given by
\begin{align*}
\dot{\phi}(n)=\frac{(-1)^n}{(n+1)^{\alpha+1}},\quad\forall n\in\N,
\end{align*}
belongs to $\mathcal{B}_N\setminus \mathcal{B}_{N+2}$.
\end{cor}

Observe that, by restriction, Theorem B implies that for all $N\geq 1$, $\mathcal{B}_{N+2}\subseteq \mathcal{B}_{N+1}\subseteq \mathcal{B}_{N}$. Corollary \ref{corB_N-B_N+2} says that one of these inclusions is strict. In order to show that both of them are, we will use the following identity.

\begin{lem}\label{lem(n+1)sum}
Let $(a_n)_{n\in\N}$ be a sequence of complex numbers. For all $n,m\in\N$,
\begin{align*}
(n+1)\sum_{k=0}^m\tbinom{m}{k}(-1)^ka_{n+k}= &\sum_{k=0}^m\tbinom{m}{k}(-1)^k(n+k+1)a_{n+k}\\
& + m\sum_{k=0}^{m-1}\tbinom{m-1}{k}(-1)^ka_{n+k+1}.
\end{align*}
\end{lem}
\begin{proof}
Observe that the identity $m\tbinom{m-1}{k}=(k+1)\tbinom{m}{k+1}$ implies
\begin{align*}
m\sum_{k=0}^{m-1}\tbinom{m-1}{k}(-1)^ka_{n+k+1}=-\sum_{k=1}^{m}k\tbinom{m}{k}(-1)^ka_{n+k}.
\end{align*}
Hence
\begin{align*}
\sum_{k=0}^m\tbinom{m}{k}(-1)^k (n+k+1)a_{n+k} + m&\sum_{k=0}^{m-1}\tbinom{m-1}{k}(-1)^ka_{n+k+1}\\
&= \sum_{k=0}^m\tbinom{m}{k}(-1)^k(n+k+1-k)a_{n+k}\\
&= (n+1)\sum_{k=0}^m\tbinom{m}{k}(-1)^ka_{n+k}.
\end{align*}
\end{proof}

%

\begin{lem}\label{lem_Bn+1BN}
For all $N\geq 1$, $\mathcal{B}_{N+1}\subsetneq \mathcal{B}_N$.
\end{lem}
\begin{proof}
Since, by Propositions \ref{prop_suff_prod_hyp} and \ref{prop_nec_prod_Cay}, $\mathcal{B}_N$ corresponds to the set of radial Schur multipliers on the product of $N$ copies of the Cayley graph of $(\Z/3\Z)\ast(\Z/3\Z)\ast(\Z/3\Z)$, we have $\mathcal{B}_{N+1}\subseteq \mathcal{B}_N$. Suppose now that there exists $N\geq 1$ such that $\mathcal{B}_{N+1}=\mathcal{B}_N$, and take $\dot{\phi}\in \mathcal{B}_N$. Then the Hankel matrix
\begin{align*}
\left((i+j+1)^{N}\mathfrak{d}_1^{N+1}\dot{\phi}(i+j)\right)_{i,j\in\N}
\end{align*}
belongs to $S_1(\ell_2(\N))$. On the other hand, applying Lemma \ref{lem(n+1)sum} to the sequence $a_n=\mathfrak{d}_1\dot{\phi}(n)$, we obtain
\begin{align*}
(n+1)\sum_{k=0}^{N}\tbinom{N}{k}(-1)^{k}\mathfrak{d}_1\dot{\phi}(n+k)&=\sum_{k=0}^N\tbinom{N}{k}(-1)^k\dot{\psi}(n+k)\\
&\qquad + N\sum_{k=0}^{N-1}\tbinom{N-1}{k}(-1)^k\mathfrak{d}_1\dot{\phi}(n+k+1),
\end{align*}
where $\dot{\psi}(n)=(n+1)\mathfrak{d}_1\dot{\phi}(n)$. 
By Lemma \ref{lemsumbinom}, this may be rewritten as
\begin{align}\label{(i+j+1)sum}
(i+j+1)\mathfrak{d}_1^{N+1}\dot{\phi}(i+j)=&\mathfrak{d}_1^{N}\dot{\psi}(i+j) + N\mathfrak{d}_1^{N}\dot{\phi}(i+j+1),
\end{align}
Since $\dot{\phi}\in \mathcal{B}_N$, by Lemma \ref{Lem_carBesov} together with Proposition \ref{Prop_Bshift}, the matrix
\begin{align*}
\left((i+j+1)^{N-1}\mathfrak{d}_1^N\dot{\phi}(i+j+1)\right)_{i,j\in\N}
\end{align*}
belongs to $S_1(\ell_2(\N))$. Hence, by \eqref{(i+j+1)sum}, the same holds for
\begin{align*}
\left((i+j+1)^{N-1}\mathfrak{d}_1^{N}\dot{\psi}(i+j)\right)_{i,j\in\N}.
\end{align*}
This says that $\dot{\psi}$ belongs to $\mathcal{B}_N$, which we have assumed is equal to $\mathcal{B}_{N+1}$. Thus
\begin{align*}
\left((i+j+1)^{N}\mathfrak{d}_1^{N+1}\dot{\psi}(i+j)\right)_{i,j\in\N}\in S_1(\ell_2(\N)).
\end{align*}
Again, by Lemmas \ref{lem(n+1)sum} and \ref{lemsumbinom},
\begin{align*}
(i+j+1)^{N+1}\mathfrak{d}_1^{N+2}\dot{\phi}(i+j)=&(i+j+1)^{N}\mathfrak{d}_1^{N+1}\dot{\psi}(i+j)\\
&\qquad +(N+1)(i+j+1)^{N}\mathfrak{d}_1^{N+1}\dot{\phi}(i+j+1),
\end{align*}
Hence,
\begin{align*}
\left((i+j+1)^{N+1}\mathfrak{d}_1^{N+2}\dot{\phi}(i+j)\right)_{i,j\in\N}\in S_1(\ell_2(\N)),
\end{align*}
which means that $\dot{\phi}\in \mathcal{B}_{N+2}$. Since $\dot{\phi}$ is arbitrary, this implies that $\mathcal{B}_N=\mathcal{B}_{N+2}$, which contradicts Corollary \ref{corB_N-B_N+2}. We conclude that $\mathcal{B}_{N+1}\neq \mathcal{B}_N$.
\end{proof}

We have proved that $\mathcal{B}_{N+1}\subsetneq \mathcal{B}_N$. The same kind of argument shows that $\mathcal{A}_{N+1}\subsetneq \mathcal{A}_N$. In order to prove that $\mathcal{C}_{N+1}\subsetneq \mathcal{C}_N$, we shall use another result of Bonsall.

\begin{thm}{\cite[Corollary 3.3]{Bon}}\label{thmBon2}
Let $(a_n)_{n\in\N}$ be a sequence of real numbers such that, for $m=0,1,2$,
\begin{align*}
(\mathfrak{d}_1^ma)_n\geq 0,\quad\forall n\in\N.
\end{align*}
Then the Hankel matrix $(a_{i+j})_{i,j\in\N}$ belongs to $S_1(\ell_2(\N))$ if and only if
\begin{align*}
\sum_{n\geq 0} a_n <\infty.
\end{align*}
\end{thm}

\begin{lem}\label{lem_d1d2}
Let $f:[0,\infty)\to\C$ be a smooth function. Define $\dot{\phi}:\N\to\C$ by $\dot{\phi}(n)=f(n)$. Then, for all $n\in\N$ and $m\geq 1$,
\begin{align}\label{form_d2}
\mathfrak{d}_2^m\dot{\phi}(n) &=(-1)^m\int_0^2\cdots\int_0^2 f^{(m)}(n+t_1+\cdots +t_m)\,dt_1\cdots dt_m.
\end{align}
\end{lem}
\begin{proof}
We proceed by induction on $m$. For $m=1$, we have
\begin{align*}
\mathfrak{d}_2\dot{\phi}(n)=f(n)-f(n+2)=-\int_0^2f'(n+t)\,dt.
\end{align*}
Now suppose that \eqref{form_d2} holds for some $m\geq 1$. Then
\begin{align*}
\mathfrak{d}_2^{m+1}\dot{\phi}(n) &= \mathfrak{d}_2^{m}\mathfrak{d}_2\dot{\phi}(n) \\
&= (-1)^m\int_0^2\cdots\int_0^2 g^{(m)}(n+t_1+\cdots +t_m)\,dt_1\cdots dt_m,
\end{align*}
with $g(t)=f(t)-f(t+2)$. Since
\begin{align*}
g^{(m)}(t)=f^{(m)}(t)-f^{(m)}(t+2) = -\int_0^2 f^{(m+1)}(t+s)\,ds,
\end{align*}
we obtain
\begin{align*}
\mathfrak{d}_2^{m+1}\dot{\phi}(n) = (-1)^{m+1}\int_0^2\cdots\int_0^2 f^{(m+1)}(n+t_1+\cdots +t_{m+1})\,dt_1\cdots dt_{m+1},
\end{align*}
which concludes the proof.
\end{proof}

\begin{lem}\label{lem_CN+1CN}
Let $N\geq 1$ and $\dot{\phi}:\N\to\C$ be given by
\begin{align*}
\dot{\phi}(n)=\begin{cases}
-\displaystyle\sum_{j=0}^{k-1}(2j+1)^{-N-\frac{1}{2}},& n=2k\\
-\displaystyle\sum_{j=1}^{k}(2j)^{-N-\frac{1}{2}},& n=2k+1.
\end{cases}
\end{align*}
Then $\dot{\phi}\in \mathcal{C}_N\setminus \mathcal{C}_{N+1}$. Moreover, if $N\geq 2$, $\dot{\phi}\in \displaystyle\bigcap_{m\geq 1}\mathcal{A}_m$.
\end{lem}
\begin{proof}
Observe that $\dot{\phi}$ is bounded and
\begin{align*}
\dot{\phi}(n)-\dot{\phi}(n+2)=(n+1)^{-N-\frac{1}{2}},\quad\forall n\in\N.
\end{align*}
Therefore, for all $m\in\N$,
\begin{align*}
(n+1)^{m}\mathfrak{d}_2\dot{\phi}(n)=(n+1)^{m-N-\frac{1}{2}},\quad\forall n\in\N.
\end{align*}
If $m<N+\frac{1}{2}$, then $a_n=(n+1)^{m-N-\frac{1}{2}}$ satisfies the hypotheses of Theorem \ref{thmBon2}, and hence, $(a_{i+j})_{i,j\in\N}\in S_1(\ell_2(\N))$ if and only if $m-N-\frac{1}{2}<-1$. Therefore
\begin{align*}
\left((i+j+1)^{N-1}\mathfrak{d}_2\dot{\phi}(i+j)\right)_{i,j\in\N}&\in S_1(\ell_2(\N))\\
\left((i+j+1)^{N}\mathfrak{d}_2\dot{\phi}(i+j)\right)_{i,j\in\N}&\notin S_1(\ell_2(\N)).
\end{align*}
This means that $\dot{\phi}\in \mathcal{C}_N\setminus \mathcal{C}_{N+1}$. Now let $\Gamma(z)=\int_0^\infty t^{z-1}e^{-1}\,dt$ be the Gamma function. By Lemma \ref{lem_d1d2} applied to the function $f(t)=(t+1)^{-N-\frac{1}{2}}$,
\begin{align*}
&\mathfrak{d}_2^{m+1}\dot{\phi}(n)\\
&=\frac{\Gamma(N+\frac{1}{2}+m)}{\Gamma(N+\frac{1}{2})}\int_0^2\cdots\int_0^2 (1+n+t_1+\cdots +t_m)^{-N-\frac{1}{2}-m}\,dt_1\cdots dt_m.
\end{align*}
for all $m\geq 0$. This gives
\begin{align*}
\frac{2^m\Gamma(\alpha+m)}{\Gamma(\alpha)}(n+1+2m)^{-N-\frac{1}{2}-m} &\leq \mathfrak{d}_2^{m+1}\dot{\phi}(n) \\
&\leq \frac{2^m\Gamma(\alpha+m)}{\Gamma(\alpha)}(n+1)^{-N-\frac{1}{2}-m},
\end{align*}
which implies that $\mathfrak{d}_2^m\dot{\phi}(n)\sim (1+n)^{-N+\frac{1}{2}-m}$ for all $m\geq 1$. Hence
\begin{align*}
(1+n)^{m}\mathfrak{d}_2^m\dot{\phi}(n)\sim (1+n)^{-N+\frac{1}{2}},\quad\forall m\geq 1.
\end{align*}
Putting $a_n=(1+n)^{-N+\frac{1}{2}}$, we get
\begin{align*}
a_n-a_{n+1}\sim (1+n)^{-N-\frac{1}{2}}.
\end{align*}
Thus
\begin{align*}
\sum_{n\geq 2} |a_{n-1}-a_n| n\log n \leq C \sum_{n\geq 2} n^{-N+\frac{1}{2}}\log n,
\end{align*}
and this is finite whenever $N\geq 2$. By Theorem \ref{thmBonsall}, $\dot{\phi}\in \mathcal{A}_m$, which concludes the proof.
\end{proof}

Now we deal with the strict inclusion $\mathcal{C}_N\subsetneq \mathcal{A}_N$. For this purpose, we shall consider the same function as in \eqref{goodphi(-1)n}, but without the alternating factor $(-1)^n$. The reason for this is that, in this case, each derivation will increase the decay rate at infinity.

\begin{lem}\label{lem_d2inCN}
Let $N\geq 2$ and $\alpha\in(0,N-1)$. Define $\dot{\phi}:\N\to\C$ by 
\begin{align*}
\dot{\phi}(n)=\frac{1}{(n+1)^\alpha},\quad\forall n\in\N.
\end{align*}
Then $\dot{\phi}\notin \mathcal{C}_N$ and $\mathfrak{d}_2^m\dot{\phi}\in \mathcal{C}_N$ for all $m>N-\alpha$.
\end{lem}
\begin{proof}
As in the proof of Lemma \ref{lem_CN+1CN}, we see that $\mathfrak{d}_2^m\dot{\phi}(n)\sim (1+n)^{-\alpha-m}$ for all $m\in\N$. Hence
\begin{align*}
\sum_{j\geq 0}(1+2j)^{N-1} \mathfrak{d}_2\dot{\phi}(2j) =\infty.
\end{align*}
This proves that $\dot{\phi}\notin \mathcal{C}_N$, since otherwise the trace of the matrix $C(N,\dot{\phi})$ would be finite. Now take $m>N-\alpha$ and observe that
\begin{align*}
\sum_{n\geq 2}&\left|(n+1)^{N-1}\mathfrak{d}_2^{m+1}\dot{\phi}(n)-n^{N-1}\mathfrak{d}_2^{m+1}\dot{\phi}(n-1)\right| n \log n \\
&\leq \sum_{n\geq 3}\left|\mathfrak{d}_2^{m+1}\dot{\phi}(n-1)\right| n^{N} \log (n-1) + \sum_{n\geq 2} \left| \mathfrak{d}_2^{m+1}\dot{\phi}(n-1)\right| n^{N} \log n\\
&\leq 2 \sum_{n\geq 2}\left|\mathfrak{d}_2^{m+1}\dot{\phi}(n-1)\right| n^N\log n\\
&\leq C \sum_{n\geq 2}n^{N-\alpha-m-1}\log n,
\end{align*}
and this is finite because $N-\alpha-m<0$. By Theorem \ref{thmBonsall}, the matrix
\begin{align*}
\left((1+i+j)^{N-1}\mathfrak{d}_2^{m+1}\dot{\phi}(i+j)\right)_{i,j\in\N}
\end{align*}
belongs to $S_1(\ell_2(\N))$, which means that $\mathfrak{d}_2^m\dot{\phi}\in \mathcal{C}_N$.
\end{proof}

\begin{cor}\label{cor_ANCN}
For all $N\geq 2$, $\mathcal{C}_N\subsetneq \mathcal{A}_N$.
\end{cor}
\begin{proof}
The inclusion $\mathcal{C}_N\subset \mathcal{A}_N$ is given by Theorems A and C, together with the fact that a product of $N$ trees is an $N$-dimensional CAT(0) cube complex. Now observe that, by definition, for every $\dot{\phi}\in\ell_\infty(\N)$, $\dot{\phi}\in \mathcal{A}_N$ if and only if $\mathfrak{d}_2^{N-1}\dot{\phi}\in \mathcal{C}_N$. If we assume that $\mathcal{A}_N=\mathcal{C}_N$, then by induction,
\begin{align*}
\dot{\phi}\in \mathcal{C}_N \iff \mathfrak{d}_2^{j(N-1)}\dot{\phi}\in \mathcal{C}_N,
\end{align*}
for all $j\in\N$. Taking $j$ big enough, this contradicts Lemma \ref{lem_d2inCN}. Therefore $\mathcal{A}_N\neq \mathcal{C}_N$.
\end{proof}

\begin{proof}[Proof of Proposition \ref{prop_inclusions}]
The inclusion $\mathcal{B}_{N+1}\subsetneq \mathcal{B}_N$ is given by Lemma \ref{lem_Bn+1BN}, and the fact that $\mathcal{A}_{N+1}\subsetneq \mathcal{A}_N$ follows analogously. Lemma \ref{lem_CN+1CN} shows that $\mathcal{C}_{N+1}\subsetneq \mathcal{C}_N$ . This proves part (a). Part (b) corresponds to Corollary \ref{cor_ANCN} together with Lemma \ref{lem_CN+1CN}. Finally, the inclusion $\mathcal{B}_N\subset \mathcal{A}_N$ is given by Theorems A and B, together with the fact that every tree is a hyperbolic graph. Moreover, by taking $\dot{\phi}(n)=(-1)^n$, we see that all the entries of the matrix $C(m,\dot{\phi})$ ($m\geq 1$) are $0$, and on the other hand, the trace of $B(N,\dot{\phi})$ does not converge. This proves that $\displaystyle\dot{\phi}\in \bigcap_{m\geq 1} \mathcal{C}_m\setminus \mathcal{B}_N$, for all $N\geq 1$. Hence $\displaystyle\bigcap_{m\geq 1} \mathcal{C}_m\nsubseteq \mathcal{B}_N$. Since $\mathcal{C}_m\subseteq \mathcal{A}_m$, this implies that $\displaystyle\bigcap_{m\geq 1} \mathcal{A}_m\nsubseteq \mathcal{B}_N$.
\end{proof}

\subsection*{Acknowledgements}
I am grateful to Mikael de la Salle for many interesting discussions and for his valuable comments on the different versions of this article. I also thank the anonymous referee for their very useful remarks and suggestions. This work was supported by CONICYT - Becas Chile (72160472) and the ANR project GAMME (ANR-14-CE25-0004).

\bibliographystyle{plain} 
\bibliography{Bibliography}

\end{document}